\documentclass[10pt,letterpaper]{article}
\makeatletter
\usepackage{pdfsync}%pdf·´ÏòËÑË÷Óõ½µÄ°ü
\usepackage{cite}
\usepackage{amscd}
\usepackage{amsmath}
\usepackage{latexsym}
\usepackage{amsfonts}
\usepackage{amssymb}
\usepackage{amsthm}
\usepackage{graphicx}
\usepackage{indentfirst}
\usepackage{enumerate}
\usepackage{amstext}
\usepackage[dvipdfm,
            pdfstartview=FitH,
            CJKbookmarks=true,
            bookmarksnumbered=true,
            bookmarksopen=true,
            colorlinks=true, %×¢Ê͵ô´ËÏîÔò½»²æÒýÓÃΪ²ÊÉ«±ß¿ò(½«colorlinksºÍpdfborder ͬʱעÊ͵ô)
            pdfborder=001,   %×¢Ê͵ô´ËÏîÔò½»²æÒýÓÃΪ²ÊÉ«±ß¿ò
            citecolor=magenta,% magenta , cyan
            linkcolor=blue,
            ]{hyperref}       % hyperref ºê°üͨ³£ÒªÇó·ÅÔÚµ¼ÑÔÇøµÄ×îºó!!!
\newtheorem{thm}{\textbf Theorem}[section]
\newtheorem{lem}{\textbf Lemma}[section]
\newtheorem{rem}{\textbf Remark}[section]
\newtheorem{cor}{\textbf Corollary}[section]

\numberwithin{equation}{section}

\newcommand{\be}{\begin{eqnarray}}
\newcommand{\ee}{\end{eqnarray}}

\newcommand{\bes}{\begin{eqnarray*}}
\newcommand{\ees}{\end{eqnarray*}}

\begin{document}
\begin{titlepage}
\title{\bf Existence and large time behavior to
the nematic liquid crystal
 equations in Besov-Morrey spaces}
%\author{Boling Guo,
%   Guoquan Qin\thanks{Corresponding Author: G. Qin} \\
% \small     Institute of Applied Physics and Computational Mathematics,\\
%\small  China Academy of Engineering Physics,  Beijing,  100088,  P. R. China.\\
%      \small Graduate School of China Academy of Engineering
%       Physics,
%        Beijing, 100088,  P. R.  China.\\
%\small        (gbl@iapcm.ac.cn, qinguoquan16@gscaep.ac.cn)
%          }

%\title{Low mach number limit of strong solutions to 3-D full Navier-Stokes equations with Dirichlet boundary condition\hspace{-4mm}}
\author{  Guoquan Qin
\\[10pt]
%\small {$^a $ Institute of Applied Physics and Computational Mathematics,
%China Academy of Engineering Physics,}\\
%\small {   Beijing,  100088,  P. R. China}\\[5pt]
\small { Graduate School of China Academy of Engineering Physics,}\\
\small {  Beijing,  100088,  P. R. China}\\[5pt]
}
%\address{Department of Mathematics and Statistics, Indian Institute of Technology Kanpur, \\Kanpur, Uttar Pradesh, India\\}
\footnotetext
{

~~~~~E-mail addresses: 690650952@qq.com}

%\author{
%Boling Guo$^1$, Guoquan Qin$^2$
%\address{
%\small  $^1$\\ Institute of Applied Physics and Computational Mathematics,
%        China Academy of Engineering Physics,  Beijing,  100088,  P. R. China.\\
%\small  $^2$Graduate School of China Academy of Engineering
%      Physics, Beijing, 100088,  P. R.  China.\\
%\small Email:gbl@iapcm.ac.cn, qinguoquan16@gscaep.ac.cn
%}
%}

\date{}
\end{titlepage}
\maketitle
\begin{abstract}
In this paper, we establish the uniquely existence of the global
mild solution to the nematic liquid crystal equations
in   Besov-Morrey spaces.
Some self-similarity and large time behavior
of the global mild solution are also investigated.

\end{abstract}

~~\noindent{ \textbf{Key words}: Global mild solution;
Besov-Morrey spaces; Nematic liquid crystal flow}

%\vspace
%%\newpage

%%%%%%%%%%%%%%%%%%%%%%%%%%%%%%%%%%%%%%%%%%%%%%%%%%%%%%%%%%%%%%%%%%%%%%%%%%%%%%%%%%%%%%%%%%%%%%%%%%%%%%%%%%%%%%%%%%%%%%%%%%%%

\section{Introduction}
\setcounter{equation}{0}
Liquid crystal describes  a state of matter in which
 the molecules may be oriented  like a crystal.
There are three main types of liquid crystals, namely, nematic, smectic and
cholesteric.
 What of frequent occurrence is the nematic type
 in which the molecules don't present
any positional order but organize in long-range orientational order.

In this paper, we study  the following
incompressible
flow of nematic liquid crystals
in $\mathbb{R}^{3}:$
\begin{equation}\label{nematic}
\left\{\begin{array}{ll}{\partial_{t} u-\Delta u+(u \cdot \nabla) u+\nabla P=-\nabla \cdot(\nabla d \odot \nabla d),} & {(x, t) \in \mathbb{R}^{3} \times(0,+\infty),} \\ {\partial_{t} d+(u \cdot \nabla) d=\Delta d+|\nabla d|^{2} d,} & {(x, t) \in \mathbb{R}^{3} \times(0,+\infty),} \\ {\nabla \cdot u=0,} & {(x, t) \in \mathbb{R}^{3} \times(0,+\infty),} \\ {\left.(u, d)\right|_{t=0}=\left(u_{0}, d_{0}\right),} & {x \in \mathbb{R}^{3},}\end{array}\right.
\end{equation}
where
$u(x, t) : \mathbb{R}^{3} \times(0,+\infty) \rightarrow \mathbb{R}^{3}$
is the unknown velocity field of the flow,
$P(x, t) : \mathbb{R}^{3} \times(0,+\infty) \rightarrow \mathbb{R}$
is a scalar pressure,
$d(x, t) : \mathbb{R}^{3} \times(0,+\infty) \rightarrow \mathbb{S}^{2}$ is the unknown (averaged)
macroscopic/continuum molecule orientation of the nematic liquid crystal flow
, where  $\mathbb{S}^{2}$ is the unit sphere in $\mathbb{R}^{3}$.
$u_{0}$ is a given initial velocity with $\nabla \cdot u_{0}=0$ in distribution sense, and $d_{0} : \mathbb{R}^{3} \rightarrow \mathbb{S}^{2}$
is a given initial liquid crystal orientation field and satisfies
$\lim _{|x| \rightarrow \infty} d_{0}(x)=\underline{d}_{0}$ with the constant unit vector
$\underline{d}_{0} \in \mathbb{S}^{2}.$ The notation $\nabla d \odot \nabla d$ denotes the $3 \times 3$ matrix whose $(i, j)$-th entry is given by $\partial_{i} d \cdot \partial_{j} d(1 \leq i, j \leq 3)$.
Note that we have set
the viscosity constants  to be 1
for simplicity.

System (\ref{nematic}) couples the forced Navier-Stokes equation with the transported flow of harmonic maps to $\mathbb{S}^{2}$.
  It has been simplified.
The original one was formulated by Ericksen and Leslie in 1960s
(see \cite{Ericksen1962, Leslie1979})  and  is one of the
most successful models for the nematic liquid crystals.

Lin \cite{Lin1989} and Lin-Liu \cite{LinLiu1995,LinLiu1996,LinLiu2000}
initiated the rigorous mathematical analysis of (\ref{nematic})
and
 considered the Ginzburg-Landau approximation of it after replacing $|\nabla d|^{2} d$  by  $\frac{1}{\epsilon^{2}}\left(1-|d|^{2}\right)d (\epsilon>0)$.
They proved the existence of global weak solutions and their partial regularities.

In 2D, Lin-Lin-Wang \cite{LinLinWang2010}  established the existence
of a global weak solution that is smooth away from at most finitely many times
for the original system (\ref{nematic})
 (see also Hong \cite{Hong2011}, Hong-Xin \cite{HongXin2012}, Hong-Li-Xin \cite{HongLiXin2014}, Huang-Lin-Wang \cite{HuangLinWang2014}, Li-Lei-Zhang \cite{LeiLiZhang2012},
Wang-Wang \cite{WangWang2014} for relevant results in dimension two, and Liu-Zhang \cite{LiuZhang2000} and Ma-Gong-Li
\cite{MaGongLi2014}).

 In 3D,
   Wen-Ding \cite{WenDing2011}  proved
the uniquely existence of local  strong solutions. Huang-Wang \cite{HuangWang2012} established a blow-up criterion of strong solutions.
 The well-posedness  for initial data
$(u_{0}, d_{0})$ with small $BMO^{-1} \times BMO$-norm and with small $L_{uloc}^{3}(\mathbb{R}^{3})$-norm was verified
 by Wang \cite{Wang2011} and Hineman-Wang \cite{HinemanWang2013}, respectively.
  Under the assumption that the
initial director field $d_{0}(\Omega) \subset \mathbb{S}_{+}^{2}$,
 Lin-Wang \cite{LinWang2016} established
 the existence of global weak solutions.

For the issue of  large time  behavior, Liu-Xu \cite{LiuXu2015} obtained an optimal decay rates for
$\|(u, \nabla d)\|_{H^{m}(\mathbb{R}^{3})}$
provided that
$(u_{0}, d_{0}) \in H^{m}(\mathbb{R}^{3}) \times H^{m+1}(\mathbb{R}^{3}, \mathbb{S}^{2})(m \geq 3)$ has sufficiently small
$\|(u_{0}, \nabla d_{0})\|_{L^{2}(\mathbb{R}^{3})}-\mathrm{norm},$
where the smallness depends on norms of higher order derivatives of
initial data.
Under the assumption that $\|u_{0}\|_{H^{1}(\mathbb{R}^{3})}+\|d-e_{3}\|_{H^{2}(\mathbb{R}^{3})}$ is sufficiently small,
Dai-Qing-Schonbek \cite{DaiQingSchonbek2012}
 and  Dai-Schonbek \cite{DaiSchonbek2014}  established an optimal decay rates in $H^{m}(\mathbb{R}^{3})$.
 Very recently, Huang-Wang-Wen \cite{HuangWangWen2019}
 consider system (\ref{nematic}) in $\mathbb{R}_{+}^{3}$
 and established some  time decay estimates under the condition
 that $(u_{0}, d_{0}) \in L_{\sigma}^{3}(\mathbb{R}_{+}^{3}) \times \dot{W}^{1,3}(\mathbb{R}_{+}^{3}, \mathbb{S}^{2})$  has small $\|(u_{0}, \nabla d_{0})\|_{L^{3}(\mathbb{R}_{+}^{3})}$ norm, which improves the conditions
on the initial data given by \cite{DaiQingSchonbek2012,DaiSchonbek2014, LiuXu2015}.
For more results on the nematic liquid crystal equations, we can refer to
\cite{LinWangPhilos2014,LinLiu1995,LiuQiao2016,LiuQiaoWang2018,XuZhang2012,LiWang2012}.

This paper aims to treat system (\ref{nematic})
in a new setting.
 We consider the framework of Besov-Morrey spaces $\mathbf{N}_{r, \lambda, q}^{-\beta}$
 which contain strongly singular functions and measures supported in either points (Diracs), filaments,
or surfaces (see e.g. [\cite{FerreiraPrecioso2011}, Remark 3.3] for more details).
Besov-Morrey spaces have been studied  in  a large number of literatures and found wide applications in analysis and
partial differential equations; see, e.g., \cite{BieWangYao2015,FerreiraPrecioso2013,ky1,Mazzucato2003,XuTan2013,YangYuan2008,YangYuan2010,YangYuan2013}.

Our motivation of this paper  is due to  Almeida-Precioso \cite{ap}
 and Yang-Fu-Sun \cite{YangFuSun2019}.
Almeida-Precioso \cite{ap}
obtained the global well-posedness
and asymptotic behavior for a semilinear heat-wave
type equation in Besov-Morrey spaces.
Yang-Fu-Sun \cite{YangFuSun2019}
established the existence and large time behavior of global mild solution to
the coupled chemotaxis-fluid equations
in Besov-Morrey spaces.
Their results are closely related to the scaling property of
the corresponding  equations and the indexes of the solution
spaces they obtained are critical.

Recall that
system (\ref{nematic}) also has a scaling property and is invariant under the following transformation
\begin{equation}\label{scaling}
\left(u_{\lambda}(x, t), P_{\lambda}(x, t), d_{\lambda}(x, t)\right) :=\left(\lambda u\left(\lambda x, \lambda^{2} t\right), \lambda^{2} P\left(\lambda x, \lambda^{2} t\right), d\left(\lambda x, \lambda^{2} t\right)\right).
\end{equation}
We say a function space is the initial critical space for system
(\ref{nematic}) if the associated norm is invariant under
the transformation $\left(u_{0}, d_{0}\right) \rightarrow\left(u_{0 \lambda}, d_{0 \lambda}\right) :=\left(\lambda u_{0}(\lambda x), d_{0}(\lambda x)\right)$ for all $\lambda>0$.

This fact leads us to consider  system (\ref{nematic})
in some critical spaces and we found the method in
\cite{YangFuSun2019} can be applied to (\ref{nematic})
to some degree.

In order to state our  results,
we first exhibit  the following  range
of the indexes and the solution spaces.

Throughout  this paper, we fix $N=3.$
Let
\begin{equation}\label{103}
\left\{\begin{array}{l}{ q_{j}, r_{j}>N-\lambda, 0 \leq \lambda<N, 1<r_{j} \leq q_{j}<\infty,\quad} \\
 {\frac{1}{q_{1}}+\frac{1}{q_{2}}<\frac{2}{N-\lambda}, \quad} \\ {\frac{2}{q_{2}}-\frac{1}{q_{1}}<\frac{1}{N-\lambda}, }\end{array}\right.
\end{equation}
and
\begin{equation}\label{104}
\left\{\begin{array}{l}{\alpha_{j}=1-\frac{N-\lambda}{q_{j}}}, \\ {\beta_{j}=1-\frac{N-\lambda}{r_{j}},}\end{array}\right.
\end{equation}
where  $j=1, 2$.

%It is important to construct the solution for the initial data
%$\left(u_{0},  d_{0}\right)$ in the Banach spaces
%$\mathcal{X}, \mathcal{Y}$ whose norms
%$\|\cdot\|_{\mathcal{X}} ,\|\cdot\|_{\mathcal{Y}}$ satisfy
%\begin{equation}
%\left\|\left(u_{0}, d_{0}\right)\right\|_{\mathcal{X} \times \mathcal{Y}}
%=\left\|\left(\left(u_{0}\right)_{\lambda},\left(d_{0}\right)_{\lambda}\right)\right\| _{\mathcal{X} \times \mathcal{Y}}.
%\end{equation}

Let $M_{q, \lambda}$
and  $\mathbf{N}_{r, \lambda, q}^{\beta}$ be the Morrey and Besov-Morrey
spaces, respectively.
For the precise definition, we can refer to Section 2.

Let
$\Theta :=\mathbb{X} \times \mathbb{Y} $
with the usual product norm
\begin{equation*}
\|(u,  d)\|_{\Theta}:=\|u\|_{\mathbb{X}}+\|d\|_{\mathbb{Y}}.
\end{equation*}

For initial data, we choose the following space
\begin{equation*}
\mathbb{E}=\dot{\mathbf{N}}_{r_{1}, \lambda, \infty}^{-\beta_{1}} \times \dot{\mathbf{N}}_{r_{2}, \lambda, \infty}^{-\beta_{2}}.
\end{equation*}
%which is invariant by (\ref{scaling}).

For a Banach space $\mathbb{F}$, let $BC_{*}([0, \infty), \mathbb{F})$
be the Banach space of all maps $\pi : [0, \infty) \rightarrow \mathbb{F}$
such that $\pi(t)$ is bounded and continuous for $t>0$ with the respect to the norm
topology of $\mathbb{F}$ and continuous at $t=0$ with respect to the weakly-star topology of $\mathbb{F}$.

For the solution, we choose the following space
\begin{equation*}
\mathbb{X} :=\left\{u : \nabla \cdot u=0, \quad u \in B C_{*}\left([0, \infty), \dot{\mathbf{N}}_{r_{1}, \lambda, \infty}^{-\beta_{1}}\right), t^{\frac{\alpha_{1}}{2}} u \in B C_{*}\left([0, \infty), M_{q_{1}, \lambda}\right)\right\}
\end{equation*}
\begin{equation*}
\begin{aligned} \mathbb{Y} :=&\left\{d : \nabla d \in B C_{*}\left([0, \infty), \dot{\mathbf{N}}_{r_{2}, \lambda, \infty}^{-\beta_{2}}\right), t^{\frac{\alpha_{2}}{2}} \nabla d \in B C_{*}\left([0, \infty), M_{q_{2}, \lambda}\right), d \in B C_{*}\left([0, \infty), L^{\infty}\right)\right\} \end{aligned}
\end{equation*}
and
\begin{equation*}
\|u\|_{\mathbb{X}} :=\sup _{t>0}\|u(t)\|_{\dot{\mathbf{N}}_{r_{1}, \lambda, \infty}^{-\beta_{1}}}+\sup _{t>0}\left(t^{\frac{\alpha_{1}}{2}}\|u(t)\|_{M_{q_{1}, \lambda}}\right)
\end{equation*}
\begin{equation*}
\|d\|_{\mathbb{Y}} :=\sup _{t>0}\|\nabla d(t)\|_{\dot{\mathbf{N}}_{r_{2}, \lambda, \infty}^{-\beta_{2}}}+\sup _{t>0}\left(t^{\frac{\alpha_{2}}{2}}\|\nabla d(t)\|_{M_{q_{2}, \lambda}}\right)+\sup _{t>0}\|d(t)\|_{L^{\infty}}.
\end{equation*}

For each $\epsilon_{0}>0,$ we say $u \in \mathbb{X}_{\epsilon_{0}},$ if $\|u\|_{\mathbb{X}} \leq C \epsilon_{0} .$ We denote $\mathbb{Y}_{\epsilon_{0}}, \mathbb{Z}_{\epsilon_{0}}$ and $\Theta_{\epsilon_{0}}$ simply as $\mathbb{X}_{\epsilon_{0}}$, respectively.

Our first result is the uniquely existence of the global mild
solution to system (\ref{nematic}).
\begin{thm}\label{thm1}
Suppose that there hold $(\ref{103})$ and $(\ref{104})$.
There exists a sufficiently small $\epsilon_{0}>0$ such that if
 \begin{equation}\label{106}
\left\|\left(u_{0},  \nabla d_{0}\right)\right\|_{\mathbb{E}}
+\left\|d_{0}-\underline{d}_{0}\right\|_{L^{\infty}}
\leq \epsilon_{0},
\end{equation}
then there exists a unique global solution to  system (\ref{nematic})  such that
$(u,  d-\underline{d}_{0}) \in \Theta_{\epsilon_{0}}$.
\end{thm}

\begin{rem}
We would like to
point out  a small difficulty in the proof of Theorem \ref{thm1}.
The small difficulty comes from the term $|\nabla d|^{2}d$.
This term leads us to include the $L^{\infty}$-norm
of $d-\underline{d}_{0}$ to $\|d\|_{\mathbb{Y}}.$
In fact, in the proof of the uniform bound part(Lemma \ref{lem301}),
there's no need to consider the  $L^{\infty}$-norm
of $d-\underline{d}_{0}$, since $|d-\underline{d}_{0}|\leq 2.$
However, in the proof of the contraction part(Lemma \ref{lem302}),
we encounter the term $|\nabla \overline{c}|^{2}c^{*}$,
where $c^{*}=\tilde{c}-\overline{c}$ denotes the difference
of two solutions in the approximation sequence.
This term deny us to repeat the process of the  proof of the uniform bound part
to prove the contraction part if $L^{\infty}$-norm is not considered.
This is also a difference between the proof of Theorem \ref{thm1} in this paper
and the proof of Theorem 1.1 in \cite{YangFuSun2019},
where $L^{\infty}$-norm must be considered in the proof of the uniform bound part.
\end{rem}

Since we work in Besov-Morrey spaces with critical indexes,
 we have the following existence
result on forward self-similar solutions to system (\ref{nematic}).

\begin{cor}\label{cor1}
Let all conditions in Theorem \ref{thm1}
hold.
If the initial data $(u_{0}, d_{0})$ satisfy
\begin{equation*}
u_{0}(x)=\lambda u_{0}\left(\lambda x\right),  \qquad d_{0}(x)=d_{0}\left(\lambda x\right)
\end{equation*}
for all $x \in \mathbb{R}^{N}$ and $\lambda>0$.
Then the  global solution $(u,  d)$ of system (\ref{nematic})
given by Theorem \ref{thm1} satisfy
\begin{equation*}
u(t, x)=\lambda u\left(\lambda^{2} t, \lambda x\right),\qquad
d(t, x)=d\left(\lambda^{2} t, \lambda x\right).
\end{equation*}
\end{cor}

We also prove an asymptotic behavior result of the global
mild solution obtained in Theorem \ref{thm1} as $t\rightarrow\infty.$
\begin{thm}\label{thm2}
Let the assumptions  in  Theorem \ref{thm1} hold, and let
$(\overline{u},  \overline{d})$ and $(\tilde{u},  \tilde{d})$ be two global
solutions for system (\ref{nematic})
given by Theorem  \ref{thm1} corresponding to  initial data
$(\overline{u}_{0},  \overline{d}_{0})$ and $(\tilde{u}_{0}, \tilde{d}_{0}),$ respectively, where
\begin{equation*}
\left\|\left(\overline{u}_{0},  \nabla \overline{d}_{0}\right)\right\|_{\mathbb{E}}
+\left\|\overline{d}_{0}-\underline{d}_{0}\right\|_{L^{\infty}}
\leq \epsilon_{0}
\end{equation*}
and
\begin{equation*}
\left\|\left(\tilde{u}_{0},  \nabla \tilde{d}_{0}\right)\right\|_{\mathbb{E}}
+\left\|\tilde{d}_{0}-\underline{d}_{0}\right\|_{L^{\infty}}
\leq \epsilon_{0}.
\end{equation*}
Then we conclude that
\begin{equation}\label{107}
\begin{array}{l}{\lim _{t \rightarrow \infty}\left(t^{\frac{\alpha_{1}}{2}}\left\|e^{t \Delta}\left(\overline{u}_{0}-\tilde{u}_{0}\right)\right\|_{q_{1}, \lambda}+t^{\frac{\alpha_{2}}{2}}\left\|\nabla e^{t \Delta}\left(\overline{d}_{0}-\tilde{d}_{0}\right)\right\|_{q_{2}, \lambda}\right)} \\ {\quad+\lim _{t \rightarrow \infty}\left(\left\|e^{t \Delta}\left(\overline{u}_{0}-\tilde{u}_{0},  \nabla \overline{d}_{0}-\nabla \tilde{d}_{0}\right)\right\|_{\mathbb{E}}+\left\|e^{t \Delta}\left(\overline{d}_{0}-\tilde{d}_{0}\right)\right\|_{L^{\infty}}\right)=0}\end{array}
\end{equation}
if and only if
\begin{equation}\label{108}
\begin{array}{c}{\lim _{t \rightarrow \infty}\left(t^{\frac{\alpha_{1}}{2}}\|(\overline{u}-\tilde{u})\|_{q_{1}, \lambda}+t^{\frac{\alpha_{2}}{2}}\|\nabla(\overline{d}-\tilde{d})\|_{q_{2}, \lambda}\right)} \\ {+\lim _{t \rightarrow \infty}\left(\|(\overline{u}-\tilde{u}, \nabla \overline{d}-\nabla \tilde{d})\|_{\mathbb{E}}+\|\overline{d}-\tilde{d}\|_{L^{\infty}}
\right)=0.}\end{array}
\end{equation}
\end{thm}

%According to Theorem \ref{thm2},
%we can derive the following asymptotically stability
%of the self-similar solution as $t\rightarrow \infty$.
%
%\begin{cor}\label{cor2}
%Let $(\overline{u},  \overline{d})$ be the self-similar solution of system $(\ref{nematic})$ ensured by Corollary \ref{cor1} with initial data
%$(\overline{u}_{0},  \overline{d}_{0})$.
%Then  there exists $\varepsilon>0$ such that for any
%$(\tilde{u}_{0},  \nabla \tilde{d}_{0}) \in \mathbb{E}$ and $\tilde{d}_{0}-\underline{d}_{0} \in L^{\infty}$ satisfying the condition $(\ref{107}),$ the corresponding mild solution
%$(\tilde{u}, \tilde{d})$ ensured by Theorem \ref{thm1} satisfies $(\ref{108})$.
%\end{cor}

The remaining of this paper is organized as follows.
In section 2, we review  some basic
properties of Morrey and Besov-Morrey space.
In section 3, we prove Theorem \ref{thm1}.
In section 4, we prove Theorem \ref{thm2}.

\section{Preliminaries}

%\subsection{Besov-Morrey space}
In this section,
the basic properties of Morrey and Besov-Morrey space is reviewed  for the reader's convenience,
more details can be found in \cite{ap,hw12,hw17, Mazzucato2003,hw25,t}.

Let $Q_{r}(x_{0})$  be the open ball in $\mathbb{R}^{N}$ centered at $x_{0}$ and with radius $r>0.$
Given two parameters $1\leq p<\infty$ and $0\leq \mu<N,$
the Morrey spaces $M_{p, \mu}=M_{p, \mu}(\mathbb{R}^{N})$ is defined to be
the set of functions $f\in L^{p}(Q_{r}(x_{0}))$ such that
\begin{equation}\label{extraeq201}
  \|f\|_{p,\mu}\triangleq
  \sup_{x_{0}\in \mathbb{R}^{N}}\sup_{r>0}r^{-\mu/p}\|f\|_{L^{p}(Q_{r}(x_{0}))}<\infty
\end{equation}
which is a Banach space endowed with norm (\ref{extraeq201}).
For $s\in \mathbb{R}$ and $1\leq p<\infty,$ the homogenous Sobolev-Morrey space
 $M_{p,\mu}^{s}=(-\Delta)^{-s/2}M_{p, \mu}$ is the Banach space
 with norm
 \begin{equation}\label{eq202}
  \|f\|_{M_{p,\mu}^{s}}=\|(-\Delta)^{s/2}f\|_{p,\mu}.\nonumber
 \end{equation}

Taking $p=1,$ we have $\|f\|_{L^{1}\left(Q_{r}\left(x_{0}\right)\right)}$
denotes the total variation of $f$ on open ball  $Q_{r}\left(x_{0}\right)$
and $M_{1, \mu}$
stands for space of signed measures. In particular, $M_{1, 0}=M$
 is  the space of finite measures. For $p>1,$ we have $M_{p, 0}=L^{p}$
  and $M_{p, 0}^{s}=\dot{H}_{p}^{s}$ is the well known Sobolev space.
  The space $L^{\infty}$ corresponds to $M_{\infty, \mu}.$
  Morrey and Sobolev-Morrey spaces present the following scaling
\begin{equation}\label{eq203}
\|f(\lambda \cdot)\|_{p, \mu}=\lambda^{-\frac{N-\mu}{p}}\|f\|_{p, \mu}\nonumber
\end{equation}
and
\begin{equation}\label{eq204}
\|f(\lambda \cdot)\|_{M_{p, \mu}^{s}}=\lambda^{s-\frac{N-\mu}{p}}\|f\|_{M_{p, \mu}^{s}},\nonumber
\end{equation}
where the  exponent $s-\frac{N-\mu}{p}$ is called scaling index and $s$ is called regularity index.
We have that
\begin{equation}\label{eq205}
(-\Delta)^{l / 2} M_{p, \mu}^{s}=M_{p, \mu}^{s-l}.\nonumber
\end{equation}
Morrey spaces contain Lebesgue and weak-$L^{p}$, with the same scaling index. Precisely, we have
the continuous proper inclusions
\begin{equation}\label{eq206}
L^{p}\left(\mathbb{R}^{N}\right) \nsubseteq \operatorname{weak}-L^{p}\left(\mathbb{R}^{N}\right) \nsubseteq M_{r, \mu}\left(\mathbb{R}^{N}\right),\nonumber
\end{equation}
where $r<p$ and $\mu=N(1-r/p)$(see e.g. \cite{hw23}).

Let $\mathcal{S}(\mathbb{R}^{N})$ and $\mathcal{S}^{'}(\mathbb{R}^{N})$ be the Schwartz space
and the tempered distributions, respectively. Let $\varphi\in \mathcal{S}(\mathbb{R}^{N})$
be nonnegative radial function such that
\begin{equation}
\operatorname{supp}(\varphi) \subset\left\{\xi \in \mathbb{R}^{N} ; \frac{1}{2}<|\xi|<2\right\}\nonumber
\end{equation}
and
\begin{equation}
\sum_{j=-\infty}^{\infty} \varphi_{j}(\xi)=1, \text { for all } \xi \neq 0,\nonumber
\end{equation}
where $\varphi_{j}(\xi)=\varphi(2^{-j}\xi)$.
Let $\phi(x)=\mathcal{F}^{-1}(\varphi)(x)$
and $\phi_{j}(x)=\mathcal{F}^{-1}\left(\varphi_{j}\right)(x)=2^{j n} \phi\left(2^{j} x\right)$
where $\mathcal{F}^{-1}$ stands for inverse Fourier transform.
For $1 \leq q<\infty, 0 \leq \mu<n$ and $s \in \mathbb{R}$,
the homogeneous Besov-Morrey space
$\dot{\mathbf{N}}_{q, \mu, r}^{s}\left(\mathbb{R}^{N}\right)$
 ($\dot{\mathbf{N}}_{q, \mu, r}^{s}$ for short)
is defined to be the set of $u\in\mathcal{S}^{'}(\mathbb{R}^{N}) $,
 modulo polynomials $\mathcal{P},$
 such that
 $\mathcal{F}^{-1} \varphi_{j}(\xi) \mathcal{F} u \in M_{q, \mu}$ for all $j \in \mathbb{Z}$ and
\begin{equation}\label{eq207}
\|u\|_{\dot{\mathbf{N}}_{q, \mu, r}^{s}}=\left\{\begin{array}{ll}{\left(\sum_{j \in \mathbb{Z}}\left(2^{j s}\left\|\phi_{j} * u\right\|_{q, \mu}\right)^{r}\right)^{\frac{1}{r}}<\infty,} & {1 \leq r<\infty,} \\ {\sup _{j \in \mathbb{Z}} 2^{j s}\left\|\phi_{j} * u\right\|_{q, \mu}<\infty,} & {r=\infty.}\end{array}\right.\nonumber
\end{equation}
The space $\dot{\mathbf{N}}_{q, \mu, r}^{s}\left(\mathbb{R}^{N}\right)$
is a Banach space and, in particular, $\dot{\mathbf{N}}_{q, 0, r}^{s}=\dot{B}_{q, r}^{s}$ (case $\mu=0$ )
corresponds to the
homogeneous Besov space. We have the real-interpolation properties
\begin{equation}\label{hweq1}
\dot{\mathbf{N}}_{q, \mu, r}^{s}=\left(M_{q, \mu}^{s_{1}}, M_{q, \mu}^{s_{2}}\right)_{\theta, r}
\end{equation}
and
\begin{equation}\label{eq208}
\dot{\mathbf{N}}_{q, \mu, r}^{s}=\left(\dot{\mathbf{N}}_{q, \mu, r_{1}}^{s_{1}}, \dot{\mathbf{N}}_{q, \mu, r_{2}}^{s_{2}}\right)_{\theta, r}
\end{equation}
for all $s_{1} \neq s_{2}, 0<\theta<1$ and $s=(1-\theta) s_{1}+\theta s_{2}.$
Here $(X, Y)_{\theta, r}$ stands for the real interpolation space
between $X$ and $Y$  constructed via the $K_{\theta, q}-$method.
Recall that $(\cdot, \cdot)_{\theta, r}$ is an exact interpolation functor of
exponent $\theta$ on the category of normed spaces.

In the next lemmas, we collect basic facts about Morrey spaces and Besov-Morrey spaces
(see \cite{ap, hw12, t}).
\begin{lem}\label{lem1}
Suppose that $s_{1}, s_{2} \in \mathbb{R}, 1 \leq p_{1}, p_{2}, p_{3}<\infty$ and $0 \leq \mu_{i}<N, i=1,2,3$.\\
(i)(Inclusion)\ \ If $\frac{N-\mu_{1}}{p_{1}}=\frac{N-\mu_{2}}{p_{2}}$ and $p_{2} \leq p_{1}$,
then
\begin{equation*}\label{eq209}
M_{p_{1}, \mu_{1}} \hookrightarrow M_{p_{2}, \mu_{2}} \quad \text { and } \dot{\mathbf{N}}_{p_{1}, \mu_{1}, 1}^{0} \hookrightarrow M_{p_{1}, \mu_{1}} \hookrightarrow \dot{\mathbf{N}}_{p_{1}, \mu_{1}, \infty}^{0}.
\end{equation*}
(ii)(Sobolev-type embedding)\ \ Let
$j=1,2$ and $p_{j}, s_{j}$ be $p_{2} \leq p_{1}, s_{1} \leq s_{2}$ such that $s_{2}-\frac{N-\mu_{2}}{p_{2}}=s_{1}-\frac{N-\mu_{1}}{p_{1}}$,
then we have
\begin{equation*}\label{eq210}
M_{p_{2}, \mu}^{s_{2}} \hookrightarrow M_{p_{1}, \mu}^{s_{1}},\left(\mu=\mu_{1}=\mu_{2}\right)
\end{equation*}
and for every $1 \leq r_{2} \leq r_{1} \leq \infty,$ we have
\begin{equation*}\label{eq211}
\dot{\mathbf{N}}_{p_{2}, \mu_{2}, r_{2}}^{s_{2}} \hookrightarrow \dot{\mathbf{N}}_{p_{1}, \mu_{1}, r_{1}}^{s_{1}} \quad \text { and } \dot{\mathbf{N}}_{p_{2}, \mu_{2}, r_{2}}^{s_{2}} \hookrightarrow \dot{B}_{\infty, r_{2}}^{s_{2}-\frac{n-\mu_{2}}{p_{2}}}.
\end{equation*}
(iii)(H\"{o}lder inequality)\ \
Let $\frac{1}{p_{3}}=\frac{1}{p_{2}}+\frac{1}{p_{1}}$ and $\frac{\mu_{3}}{p_{3}}=\frac{\mu_{2}}{p_{2}}+\frac{\mu_{1}}{p_{1}}$.
If $f_{j} \in M_{p_{j}, \mu_{j}}$ with $j=1,2,$ then $f_{1} f_{2} \in M_{p_{3}, \mu_{3}}$ and
\begin{equation*}\label{eq212}
\left\|f_{1} f_{2}\right\|_{p_{3}, \mu_{3}} \leq\left\|f_{1}\right\|_{p_{1}, \mu_{1}}\left\|f_{2}\right\|_{p_{2}, \mu_{2}}.
\end{equation*}
\end{lem}

%The following Lemma
%is an estimate for certain multiplier operators on
%$\mathcal{M}_{q, \mu}^{s}$.
%\begin{lem}\label{lem2}
%Let $m, s \in \mathbb{R}$ and $0 \leq \mu<n$ and $P(\xi) \in C^{[n / 2]+1}\left(\mathbb{R}^{n} \backslash\{0\}\right)$.
%Assume that there is $A>0$ such that
%\begin{equation}\label{eq213}
%\left|\frac{\partial^{k} P}{\partial \xi^{k}}(\xi)\right| \leq A|\xi|^{m-|k|}\nonumber
%\end{equation}
%for all $k \in(\mathbb{N} \cup\{0\})^{n}$ with $|k| \leq[n / 2]+1$ and for all $\xi \neq 0$.
%Then, the multiplier operator $P(D) f=\mathcal{F}^{-1} P(\xi) \mathcal{F} f \text { on } \mathcal{S}^{\prime} / \mathcal{P}$ on $\mathcal{S}^{\prime} / \mathcal{P}$ satisfies the estimate
%\begin{equation}\label{eq214}
%\|P(D) f\|_{\mathcal{M}_{q, \mu}^{s-m}} \leq C A\|f\|_{\mathcal{M}_{q, \mu}^{s}},\ \ (1<q<\infty)\nonumber
%\end{equation}
%where $C > 0$ is a constant independent of $f$,
%and the set  $\mathcal{S}^{\prime} / \mathcal{P}$ consists in equivalence classes in $\mathcal{S}^{\prime}$  modulo
%polynomials with $n$ variables.
%\end{lem}

Set $\alpha=1$ in [\cite{ap}, Lemma 3.1],
we have the following decay estimates about
the heat semi-group in the Sobolev-Morrey
or Besov-Morrey space.
\begin{lem}\label{lem3}
Let $s, \beta \in \mathbb{R}, 1<p \leq q\leq\infty, 0 \leq \mu<N,$ and $(\beta-s)+\frac{N-\mu}{p}-\frac{N-\mu}{q}<2$ where $\beta \geq s.$\\
\ \ There exists $C > 0$ such that
\begin{equation}
\left\|e^{t\Delta} f\right\|_{M_{q, \mu}^{\beta}} \leq C t^{-\frac{1}{2}(\beta-s)-\frac{1}{2}\left(\frac{N-\mu}{p}-\frac{N-\mu}{q}\right)}\|f\|_{M_{p, \mu}^{s}}\nonumber
\end{equation}
for every $t>0$ and $f \in M_{p, \mu}^{s}$.
%(ii)\ \ Let $r \in[1, \infty],$ there exists $C>0$ such that
%\begin{equation}
%\left\|e^{t\Delta} f\right\|_{\dot{\mathcal{N}}_{q, \mu, r}^{\beta}} \leq C t^{-\frac{1}{2}(\beta-s)-\frac{1}{2}\left(\frac{n-\mu}{p}-\frac{n-\mu}{q}\right)}\|f\|_{\dot{\mathcal{N}}_{p, \mu, r}^{s}}\nonumber
%\end{equation}
%for every $f \in \mathcal{S}^{\prime} / \mathcal{P}$ and $t>0$.\\
%(iii)\ \ Let $r \in[1, \infty]$ and $\beta>s,$ there exists $C>0$ such that
%\begin{equation}
%\left\|e^{t\Delta} f\right\|_{\dot{\mathcal{N}}_{q, \mu, 1}^{\beta}} \leq C t^{-\frac{1}{2}(\beta-s)-\frac{1}{2}\left(\frac{n-\mu}{p}-\frac{n-\mu}{q}\right)}\|f\|_{\dot{\mathcal{N}}_{p, \mu, r}^{s}}\nonumber
%\end{equation}
%for every $f \in \mathcal{S}^{\prime} / \mathcal{P}$.
\end{lem}
The following Lemma can be found in \cite{Mazzucato2003}.
\begin{lem}\label{lemy23}
If $1 \leq p, q \leq \infty, s>0,0 \leq \lambda<N,$ then $f \in \dot{\mathbf{N}}_{p, \lambda, q}^{-2 s}$ if and only if
\begin{equation*}
\left\{\begin{array}{ll}{\left[\int_{0}^{\infty}\left(t^{s}\left\|e^{t \Delta} f\right\|_{p, \lambda}\right)^{q} \frac{d t}{t}\right]^{\frac{1}{q}}<\infty,} & {\text { if } 1 \leq q<\infty,} \\ {\sup _{t>0}\left(t^{s}\left\|e^{t \Delta} f\right\|_{p, \lambda}\right)<\infty,} & {\text { if } q=\infty.}\end{array}\right.
\end{equation*}
\end{lem}

\section{Proof of Theorem \ref{thm1} -- global mild solution}
The proof of Theorem \ref{thm1}
is a consequence of the following  Lemmas
\ref{301} and \ref{302}.
We will prove
it by a fixed point argument.

Let $\mathbb{P} \triangleq I-\nabla \Delta^{-1} \mathrm{div}$ be the Leray projection operator.
Denote $c \triangleq d-\underline{d}_{0}$ and $c_{0} \triangleq d_{0}-\underline{d}_{0}.$ Then we can rewrite system $(\ref{nematic})$ as
\begin{equation}\label{l315}
\left\{\begin{array}{l}{\partial_{t} u-\Delta u=-\mathbb{P}[u \cdot \nabla u+\operatorname{div}(\nabla c \odot \nabla c)],}
 \\ {\partial_{t} c-\Delta c=-u \cdot \nabla c+|\nabla c|^{2} c+|\nabla c|^{2} \underline{d}_{0},}
 \\ {\left.u\right|_{t=0}=u_{0}(x), \quad\left.c\right|_{t=0}=c_{0}(x),}\end{array}\right.
\end{equation}
where the initial data satisfying the following far field behavior
\begin{equation}\label{l316}
u_{0} \rightarrow 0, \quad c_{0} \rightarrow 0 \quad \text { as }|x| \rightarrow \infty.
\end{equation}
By the Duhamel principle, we can express a solution
$(u, c)$ of $(\ref{l315})$ and $(\ref{l316})$ in the integral form:
\begin{equation}\label{l317}
\left\{\begin{array}{l}{u(t)=e^{t \Delta} u_{0}-\int_{0}^{t} e^{(t-s) \Delta} \mathbb{P}[u \cdot \nabla u+\operatorname{div}(\nabla c \odot \nabla c)](s) \mathrm{d} s,} \\
{c(t)=e^{t \Delta} c_{0}+\int_{0}^{t} e^{(t-s) \Delta}\left[-u \cdot \nabla c+|\nabla c|^{2} c+|\nabla c|^{2} \underline{d}_{0}\right](s) \mathrm{d} s.}\end{array}\right.
\end{equation}

Define the  map
\begin{equation}\label{301}
(u,  c)=\mathbb{T}(\tilde{u}, \tilde{c})=(\mathbb{T}_{1}(\tilde{u},  \tilde{c}), \mathbb{T}_{2}(\tilde{u},  \tilde{c}))
\end{equation}
with
\begin{equation}\label{302}
\left\{\begin{array}{l}{u(t)
=\mathbb{T}_{1}(\tilde{u},  \tilde{c})
=e^{t \Delta} u_{0}-\int_{0}^{t} e^{(t-s) \Delta} \mathbb{P}[\tilde{u} \cdot \nabla \tilde{u}+\operatorname{div}(\nabla \tilde{c} \odot \nabla \tilde{c})](s) \mathrm{d}s,} \\
 {c(t)=\mathbb{T}_{2}(\tilde{u},  \tilde{c})
 =e^{t \Delta} c_{0}+\int_{0}^{t} e^{(t-s) \Delta}\left[-\tilde{u} \cdot \nabla \tilde{c}+|\nabla \tilde{c}|^{2} \tilde{c}+|\nabla \tilde{c}|^{2} \underline{d}_{0}\right](s) \mathrm{d}s.}\end{array}\right.
\end{equation}
Then we have
\begin{lem}\label{lem301}
 Given a constant $\epsilon_{0}>0$ small enough, the initial data
$\left(u_{0}, c_{0} \right)$ satisfies (\ref{106}) and $(\tilde{u},  \tilde{c}) \in \Theta_{\epsilon_{0}},$ then the solution of $(\ref{302})$ satisfies
\begin{equation*}
(u, c)=\mathbb{T}(\tilde{u}, \tilde{c}) \in \Theta_{\epsilon_{0}}.
\end{equation*}
\end{lem}

\begin{proof}
%We first estimate
%\begin{equation*}
%\left\|\int_{0}^{t} e^{(t-s) \Delta} \mathbb{P}[\tilde{u} \cdot \nabla \tilde{u}+\operatorname{div}(\nabla \tilde{c} \odot \nabla \tilde{c})](s) \mathrm{d} s\right\|_{\mathbb{X}}.
%\end{equation*}
 Let's first consider the  $\dot{\mathbf{N}}_{r_{1}, \lambda, \infty}^{-\beta_{1}}$-norm of
\begin{equation*}
\int_{0}^{t} e^{(t-s) \Delta} \mathbb{P}[\tilde{u} \cdot \nabla \tilde{u}+\operatorname{div}(\nabla \tilde{c} \odot \nabla \tilde{c})](s) \mathrm{d} s.
\end{equation*}
 Lemma \ref{lemy23} and the boundedness of $\mathbb{P}$ in Morrey spaces
 lead to
\begin{eqnarray}\label{303}
&&\left\|\int_{0}^{t} e^{(t-s) \Delta} \mathbb{P}[\tilde{u} \cdot \nabla \tilde{u}+\operatorname{div}(\nabla \tilde{c} \odot \nabla \tilde{c})](s) \mathrm{d} s\right\|_{\dot{\mathbf{N}}_{r_{1}, \lambda, \infty}^{-\beta_{1}}}\nonumber\\
&&=\sup _{s>0}\left[s^{\beta_{1} / 2} \left\| e^{s \Delta} \int_{0}^{t} e^{(t-\tau) \Delta}
 \mathbb{P}[\tilde{u} \cdot \nabla \tilde{u}+\operatorname{div}(\nabla \tilde{c} \odot \nabla \tilde{c})](\cdot, \tau) \mathrm{d} \tau\right\|_{r_{1}, \lambda}\right]
\nonumber\\
&&\leq \int_{0}^{t} \sup _{s>0}\left[s^{\beta_{1} / 2}\left\|e^{s \Delta} e^{(t-\tau) \Delta}
 \mathbb{P}[\tilde{u} \cdot \nabla \tilde{u}+\operatorname{div}(\nabla \tilde{c} \odot \nabla \tilde{c})](\cdot, \tau)\right\|_{r_{1}, \lambda}\right] \mathrm{d} \tau.
\end{eqnarray}
For $0<s \leq t-\tau,$ we have by Lemma \ref{lem1} that
\begin{eqnarray}\label{304}
&&\sup _{0<s\leq t-\tau}\left[s^{\beta_{1} / 2}\left\|e^{s \Delta} e^{(t-\tau) \Delta}
 \mathbb{P}[\tilde{u} \cdot \nabla \tilde{u}+\operatorname{div}(\nabla \tilde{c} \odot \nabla \tilde{c})](\cdot, \tau)\right\|_{r_{1}, \lambda}\right]\nonumber\\
 &&\leq (t-\tau)^{\frac{\beta_{1}}{2}}(t-\tau)^{-\frac{1}{2}-\frac{N-\lambda}{2}
 \left(\frac{2}{q_{1}}-\frac{1}{r_{1}}\right)}\|\tilde{u} \otimes \tilde{u}\|_{\frac{q_{1}}{2}, \lambda}\nonumber\\
 &&\quad +(t-\tau)^{\frac{\beta_{1}}{2}}(t-\tau)^{-\frac{1}{2}-\frac{N-\lambda}{2}\left(\frac{2}{q_{2}}-\frac{1}{r_{1}}\right)}
 \|\nabla\tilde{c} \odot \nabla\tilde{c}\|_{ \frac{q_{2}}{2}, \lambda}\nonumber\\
 &&\leq (t-\tau)^{-\frac{N-\lambda}{q_{1}}}\|\tilde{u}\|_{q_{1}, \lambda}^{2}
 +(t-\tau)^{-\frac{N-\lambda}{q_{2}}}\|\nabla\tilde{c}\|_{q_{2}, \lambda}^{2}.
\end{eqnarray}

%\begin{equation}
%\begin{array}{c}{\lim _{t \rightarrow \infty}\left(t^{\frac{\alpha_{1}}{2}}\|(\overline{u}-\tilde{u})\|_{q_{1}, \lambda}+t^{\frac{\alpha_{2}}{2}}\|(\overline{n}-\tilde{n})\|_{q_{2}, \lambda}+t^{\frac{\alpha_{3}}{2}}\|\nabla(\overline{c}-\tilde{c})\|_{q_{3}, \lambda}\right)} \\ {+\lim _{t \rightarrow \infty}\left(\|(\overline{u}-\tilde{u}, \overline{n}-\tilde{n}, \nabla \overline{c}-\nabla \tilde{c})\|_{\mathbb{E}}+\|\overline{c}-\tilde{c}\|_{L^{\infty}}\right)=0}\end{array}
%\end{equation}
%

For $s>t-\tau,$ note that $(t-\tau+s) / 2<s<t-\tau+s,$ so we have
from Lemma  \ref{lem1}  that
\begin{eqnarray}\label{305}
&&\sup _{s> t-\tau}\left[s^{\beta_{1} / 2}\left\|e^{s \Delta} e^{(t-\tau) \Delta}
 \mathbb{P}[\tilde{u} \cdot \nabla \tilde{u}+\operatorname{div}(\nabla \tilde{c} \odot \nabla \tilde{c})](\cdot, \tau)\right\|_{r_{1}, \lambda}\right]\nonumber\\
 &&\leq
 \sup_{s>t-\tau}\left[s^{\frac{\beta_{1}}{2}}(s+t-\tau)^{-\frac{1}{2}-\frac{N-\lambda}{2}\left(\frac{2}{q_{1}}-\frac{1}{r_{1}}\right)}\right]\|(\tilde{u} \otimes \tilde{u})(\cdot, \tau)\| _{\frac{q_{1}}{2}, \lambda}\nonumber\\
 &&\quad +\sup_{s>t-\tau}\left[s^{\frac{\beta_{1}}{2}}(s+t-\tau)^{-\frac{1}{2}-\frac{N-\lambda}{2}\left(\frac{2}{q_{2}}-\frac{1}{r_{1}}\right)}\right]\|(\nabla\tilde{c} \odot \nabla\tilde{c})(\cdot, \tau)\| _{\frac{q_{2}}{2}, \lambda}\nonumber\\
 &&\leq (t-\tau)^{-\frac{N-\lambda}{q_{1}}} \sup _{s>t-\tau}\left(1+\frac{s}{t-\tau}\right)^{-\frac{N-\lambda}{q_{1}}}\|\tilde{u}(\cdot, \tau)\|_{q_{1}, \lambda}^{2}\nonumber\\
 &&\quad +(t-\tau)^{-\frac{N-\lambda}{q_{2}}} \sup _{s>t-\tau}\left(1+\frac{s}{t-\tau}\right)^{-\frac{N-\lambda}{q_{2}}}\|\nabla\tilde{c}(\cdot, \tau)\|_{q_{2}, \lambda}^{2}\nonumber\\
 &&\leq
 (t-\tau)^{-\frac{N-\lambda}{q_{1}}}\|\tilde{u}(\cdot, \tau)\|_{q_{1}, \lambda}^{2}
 +(t-\tau)^{-\frac{N-\lambda}{q_{2}}}\|\nabla\tilde{c}(\cdot, \tau)\|_{q_{2}, \lambda}^{2}.
\end{eqnarray}

Substituting $(\ref{304})$ and $(\ref{305})$ into $(\ref{303})$
 and using $q_{1}>N-\lambda, q_{2}>N-\lambda$  yield
 \begin{eqnarray}\label{306}
&&\left\|\int_{0}^{t} e^{(t-s) \Delta} \mathbb{P}[\tilde{u} \cdot \nabla \tilde{u}+\operatorname{div}(\nabla \tilde{c} \odot \nabla \tilde{c})](s) \mathrm{d} s\right\|_{\dot{\mathbf{N}}_{r_{1}, \lambda, \infty}^{-\beta_{1}}}\nonumber\\
&&\lesssim \left\{\left[\sup _{t>0}\left(t^{\frac{\alpha_{1}}{2}}\|\tilde{u}(\cdot, t)\|_{q_{1}, \lambda}\right)\right]^{2}
+\left[\sup _{t>0}\left(t^{\frac{\alpha_{2}}{2}}\|\nabla\tilde{c}(\cdot, t)\|_{q_{2}, \lambda}\right)\right]^{2}\right\}\nonumber\\
&&\quad\times \int_{0}^{t}(t-\tau)^{-\frac{N-\lambda}{q_{1}}} \tau^{-\alpha_{1}}+(t-\tau)^{-\frac{N-\lambda}{q_{2}}} \tau^{-\alpha_{2}} d \tau\nonumber\\
&&\lesssim [B(\alpha_{1}, 1-\alpha_{1})+B(\alpha_{2}, 1-\alpha_{2})] (\|\tilde{u}\|_{\mathbb{X}}^{2} + \|\tilde{c}\|_{\mathbb{Y}}^{2})\nonumber\\
&&\lesssim \|\tilde{u}\|_{\mathbb{X}}^{2} + \|\tilde{c}\|_{\mathbb{Y}}^{2},
\end{eqnarray}
where $B(x, y) :=\int_{0}^{1}(1-t)^{x-1} t^{y-1} d t,$ for $x>0, y>0$.

Next, we calculate the $M_{q_{1}, \lambda}$-norm of
\begin{equation*}
  \int_{0}^{t} e^{(t-s) \Delta} \mathbb{P}[\tilde{u} \cdot \nabla \tilde{u}+\operatorname{div}(\nabla \tilde{c} \odot \nabla \tilde{c})](s) \mathrm{d} s.
\end{equation*}
Using Lemmas \ref{lem1}, \ref{lem3} and \ref{lemy23} and  noting that
$2 / q_{2}-1 / q_{1}<1 / (N-\lambda), q_{1}>N-\lambda,$ we have
\begin{eqnarray}
&&\left\|\int_{0}^{t} e^{(t-\tau) \Delta} \mathbb{P}[\tilde{u} \cdot \nabla \tilde{u}+\operatorname{div}(\nabla \tilde{c} \odot \nabla \tilde{c})](\cdot, \tau) \mathrm{d} \tau\right\|_{q_{1}, \lambda}\nonumber\\
&&\lesssim \int_{0}^{t}(t-\tau)^{-1 / 2-\frac{N-\lambda}{2}\left(\frac{2}{q_{1}}-\frac{1}{q_{1}}\right)}\|\tilde{u} \otimes \tilde{u}\|_{q_{1} / 2, \lambda}\nonumber\\
&&\quad+
(t-\tau)^{-1 / 2-\frac{N-\lambda}{2}\left(\frac{2}{q_{2}}-\frac{1}{q_{1}}\right)}\|\nabla \tilde{c} \odot \nabla \tilde{c}\|_{q_{2} / 2, \lambda}\mbox{d}\tau\nonumber\\
&&\lesssim \left\{\left[\sup _{t>0}\left(t^{\frac{\alpha_{1}}{2}}\|\tilde{u}(\cdot, t)\|_{q_{1}, \lambda}\right)\right]^{2}+\left[\sup _{t>0}\left(t^{\frac{\alpha_{2}}{2}}\|\nabla \tilde{c}(\cdot, t)\|_{q_{2}, \lambda}\right)\right]^{2}\right\}\nonumber\\
&&\quad\times \int_{0}^{t}(t-\tau)^{-1 / 2-\frac{N-\lambda}{2}\left(\frac{2}{q_{1}}-\frac{1}{q_{1}}\right)} \tau^{-\alpha_{1}}+(t-\tau)^{-1 / 2-\frac{N-\lambda}{2}\left(\frac{2}{q_{2}}-\frac{1}{q_{1}}\right)} \tau^{-\alpha_{2}}\mbox{d}\tau\nonumber\\
&&\lesssim t^{-\frac{1}{2}+\frac{N-\lambda}{2 q_{1}}}(\|\tilde{u}\|_{\mathbb{X}}^{2}+\| \tilde{c}\|_{\mathbb{Y}}^{2}).
\end{eqnarray}

Next, We calculate the $\dot{\mathbf{N}}_{r_{2}, \lambda, \infty}^{-\beta_{2}}$ -norm of $\nabla\int_{0}^{t} e^{(t-\tau) \Delta}\left[-\tilde{u} \cdot \nabla \tilde{c}+|\nabla \tilde{c}|^{2} \tilde{c}+|\nabla \tilde{c}|^{2} \underline{d}_{0}\right](\tau) \mathrm{d}\tau$.
By Lemma \ref{lemy23},
\begin{eqnarray}
&&\left\|\nabla \int_{0}^{t} e^{(t-\tau) \Delta}\left[-\tilde{u} \cdot \nabla \tilde{c}+|\nabla \tilde{c}|^{2} \tilde{c}+|\nabla \tilde{c}|^{2} \underline{d}_{0}\right](\tau) \mathrm{d}\tau\right\|_{\dot{\mathbf{N}}_{r_{2}, \lambda, \infty}^{-\beta_{2}}}\nonumber\\
&&=\sup _{s>0}\left[s^{\beta_{2} / 2} \left\| \nabla e^{s \Delta} \int_{0}^{t} e^{(t-\tau) \Delta}(-\tilde{u} \cdot \nabla \tilde{c}+|\nabla \tilde{c}|^{2} \tilde{c}+|\nabla \tilde{c}|^{2} \underline{d}_{0})(\cdot, \tau)\mathrm{d} \tau\right\|_{r_{2}, \lambda}\right]\nonumber\\
&&\leq \int_{0}^{t} \sup _{s>0}\left[s^{\beta_{2} / 2}\left\|e^{s \Delta} \nabla e^{(t-\tau) \Delta}(-\tilde{u} \cdot \nabla \tilde{c}+|\nabla \tilde{c}|^{2} \tilde{c}+|\nabla \tilde{c}|^{2} \underline{d}_{0})(\cdot, \tau)\right\|_{r_{2}, \lambda}\right]\mathrm{d}\tau.
\end{eqnarray}

Employing the fact that $|\tilde{c}|\leq 2$
and $|\underline{d}_{0}|\leq 1$,
we obtain by Lemmas \ref{lem1} and \ref{lem3} that
\begin{eqnarray}
&&\sup _{0<s \leq t-\tau}\left[s^{\frac{\beta_{2}}{2}}\left\|e^{s \Delta} \nabla e^{(t-\tau) \Delta}(-\tilde{u} \cdot \nabla \tilde{c}+|\nabla \tilde{c}|^{2} \tilde{c}+|\nabla \tilde{c}|^{2} \underline{d}_{0})(\cdot, \tau)\right\|_{r_{2}, \lambda}\right]\nonumber\\
&&\lesssim
(t-\tau)^{\frac{\beta_{2}}{2}}(t-\tau)^{-\frac{1}{2}-\frac{N-\lambda}{2}\left(\frac{1}{q_{1}}+\frac{1}{q_{2}}-\frac{1}{r_{2}}\right)}\|\tilde{u}(\cdot, \tau) \nabla \tilde{c}(\cdot, \tau)\|_{\frac{q_{1} q_{2}}{q_{1}+q_{2}}, \lambda}\nonumber\\
&&\quad+(t-\tau)^{\frac{\beta_{2}}{2}}(t-\tau)^{-\frac{1}{2}-\frac{N-\lambda}{2}\left(\frac{2}{q_{2}}-\frac{1}{r_{2}}\right)}\| |\nabla \tilde{c}|^{2}(\cdot, \tau) \|_{\frac{q_{2}}{2}, \lambda}\nonumber\\
&&\lesssim
(t-\tau)^{-\frac{N-\lambda}{2 q_{1}}-\frac{N-\lambda}{2 q_{2}}}\|\tilde{u}(\cdot, \tau)\|_{q_{1}, \lambda}\|\nabla \tilde{c}(\cdot, \tau)\|_{q_{2}, \lambda}
+(t-\tau)^{-\frac{N-\lambda}{ q_{2}}}\|\nabla \tilde{c}(\cdot, \tau)\|_{q_{2}, \lambda}^{2}.
\end{eqnarray}
For $s>t-\tau,$ note that $(t-\tau+s) / 2<s<t-\tau+s,$
one obtains
\begin{eqnarray}
&&\sup _{s > t-\tau}\left[s^{\frac{\beta_{2}}{2}}\left\|e^{s \Delta} \nabla e^{(t-\tau) \Delta}(-\tilde{u} \cdot \nabla \tilde{c}+|\nabla \tilde{c}|^{2} \tilde{c}+|\nabla \tilde{c}|^{2} \underline{d}_{0})(\cdot, \tau)\right\|_{r_{2}, \lambda}\right]\nonumber\\
&&\lesssim
s^{\frac{\beta_{2}}{2}}(s+t-\tau)^{-\frac{1}{2}-\frac{N-\lambda}{2}\left(\frac{1}{q_{1}}+\frac{1}{q_{2}}-\frac{1}{r_{2}}\right)}\|\tilde{u}(\cdot, \tau) \nabla \tilde{c}(\cdot, \tau)\|_{\frac{q_{1} q_{2}}{q_{1}+q_{2}}, \lambda}\nonumber\\
&&\quad+s^{\frac{\beta_{2}}{2}}(s+t-\tau)^{-\frac{1}{2}-\frac{N-\lambda}{2}\left(\frac{2}{q_{2}}-\frac{1}{r_{2}}\right)}\| |\nabla \tilde{c}|^{2}(\cdot, \tau)\|_{\frac{ q_{2}}{2}, \lambda}\nonumber\\
&&\lesssim
(t-\tau)^{-\frac{N-\lambda}{2 q_{1}}-\frac{N-\lambda}{2 q_{2}}}
\sup _{s>t-\tau}\left(1+\frac{s}{t-\tau}\right)^{-\frac{N-\lambda}{2 q_{1}}-\frac{N-\lambda}{2 q_{2}}}\|\tilde{u}(\cdot, \tau)\|_{q_{1}, \lambda}\|\nabla \tilde{c}(\cdot, \tau)\|_{q_{2}, \lambda}\nonumber\\
&&\quad+(t-\tau)^{-\frac{N-\lambda}{ q_{2}}}
\sup _{s>t-\tau}\left(1+\frac{s}{t-\tau}\right)^{-\frac{N-\lambda}{q_{2}}}\|\nabla \tilde{c}(\cdot, \tau)\|_{q_{2}, \lambda}^{2}.
\end{eqnarray}
Note that $\frac{1}{q_{1}}+\frac{1}{q_{2}}<\frac{2}{N-\lambda}, \frac{1}{q_{2}}<\frac{1}{N-\lambda},$ we thus obtain
\begin{eqnarray}\label{310}
&&\left\|\nabla \int_{0}^{t} e^{(t-\tau) \Delta}\left[-\tilde{u} \cdot \nabla \tilde{c}+|\nabla \tilde{c}|^{2} \tilde{c}+|\nabla \tilde{c}|^{2} \underline{d}_{0}\right](\tau) \mathrm{d}\tau\right\|_{\dot{\mathbf{N}}_{r_{2}, \lambda, \infty}^{-\beta_{2}}}\nonumber\\
&&\leq \left\{\left[\sup _{t>0}\left(t^{\frac{\alpha_{1}}{2}}\|\tilde{u}(\cdot, t)\|_{q_{1}, \lambda}\right)\right]\left[\sup _{t>0}\left(t^{\frac{\alpha_{2}}{2}}\|\nabla \tilde{c}(\cdot, t)\|_{q_{2}, \lambda}\right)\right]
+\left[\sup _{t>0}\left(t^{\alpha_{2}}\|\nabla \tilde{c}(\cdot, t)\|_{q_{2}, \lambda}^{2}\right)\right]\right\}\nonumber\\
&&\quad\times \int_{0}^{t}(t-\tau)^{-\frac{N-\lambda}{2 q_{1}}-\frac{N-\lambda}{2 q_{2}}} \tau^{-\frac{\alpha_{1}}{2}} \tau^{-\frac{\alpha_{2}}{2}}+(t-\tau)^{-\frac{N-\lambda}{ q_{2}}} \tau^{-\alpha_{2}} \mathrm{d}\tau\nonumber\\
&&\lesssim \|u\|_{\mathbb{X}}\|\tilde{c}\|_{\mathbb{Y}}
+\|\tilde{c}\|_{\mathbb{Y}}^{2}.
\end{eqnarray}
Next, we calculate the $M_{q_{2}, \lambda}$-norm of
$\nabla\int_{0}^{t} e^{(t-\tau) \Delta}\left[-\tilde{u} \cdot \nabla \tilde{c}+|\nabla \tilde{c}|^{2} \tilde{c}+|\nabla \tilde{c}|^{2} \underline{d}_{0}\right](\tau) \mathrm{d}\tau$.
\begin{eqnarray}\label{311}
  &&\left\|\nabla \int_{0}^{t} e^{(t-\tau) \Delta}(-\tilde{u} \cdot \nabla \tilde{c}+|\nabla \tilde{c}|^{2} \tilde{c}+|\nabla \tilde{c}|^{2} \underline{d}_{0})(\cdot, \tau) \mathrm{d} \tau\right\|_{q_{2}, \lambda}\nonumber\\
  &&\lesssim
  \int_{0}^{t}(t-\tau)^{-\frac{1}{2}-\frac{N-\lambda}{2}\left(\frac{1}{q_{1}}+\frac{1}{q_{2}}-\frac{1}{q_{2}}\right)}\|\tilde{u} \nabla \tilde{c}\|_{\frac{q_{1} q_{2}}{q_{1}+q_{2}}, \lambda}
  +(t-\tau)^{-\frac{1}{2}-\frac{N-\lambda}{2}\left(\frac{2}{q_{2}}-\frac{1}{q_{2}}\right)}\| |\nabla \tilde{c}|^{2}\|_{\frac{ q_{2}}{2}, \lambda}\mbox{d}\tau\nonumber\\
  &&\lesssim
  \left\{\left[\sup _{t>0}\left(t^{\frac{\alpha_{1}}{2}}\|\tilde{u}\|_{q_{1}, \lambda}\right)\right]\left[\sup _{t>0}\left(t^{\frac{\alpha_{2}}{2}}\|\nabla \tilde{c}\|_{q_{2}, \lambda}\right)\right]
  +\left[\sup _{t>0}\left(t^{\alpha_{2}}\|\nabla \tilde{c}\|_{q_{2}, \lambda}^{2}\right)\right]\right\}\nonumber\\
  &&\quad\times \int_{0}^{t}(t-\tau)^{-\frac{1}{2}-\frac{N-\lambda}{2}\left(\frac{1}{q_{1}}+\frac{1}{q_{2}}-\frac{1}{q_{2}}\right)} \tau^{-\frac{\alpha_{1}}{2}} \tau^{-\frac{\alpha_{2}}{2}}+
  (t-\tau)^{-\frac{1}{2}-\frac{N-\lambda}{2}\left(\frac{2}{q_{2}}-\frac{1}{q_{2}}\right)} \tau^{-\alpha_{2}} \mbox{d}\tau\nonumber\\
  &&\lesssim t^{-\frac{1}{2}+\frac{N-\lambda}{2 q_{2}}}(\|\tilde{u}\|_{\mathbb{X}}\|\tilde{c}\|_{\mathbb{Y}}
  +\|\tilde{c}\|_{\mathbb{Y}}^{2}).
\end{eqnarray}

We calculate $L^{\infty}$-norm of
$\int_{0}^{t} e^{(t-\tau) \Delta}\left[-\tilde{u} \cdot \nabla \tilde{c}+|\nabla \tilde{c}|^{2} \tilde{c}+|\nabla \tilde{c}|^{2} \underline{d}_{0}\right](\tau) \mathrm{d}\tau$.
Note that $q_{1}>N-\lambda, q_{2}>N-\lambda,$ we get
\begin{eqnarray}\label{312}
&&\left\|\int_{0}^{t} e^{(t-\tau) \Delta}\left[-\tilde{u} \cdot \nabla \tilde{c}+|\nabla \tilde{c}|^{2} \tilde{c}+|\nabla \tilde{c}|^{2} \underline{d}_{0}\right](\cdot, \tau) \mathrm{d} \tau\right\|_{L^{\infty}}\nonumber\\
&&\lesssim
\int_{0}^{t}(t-\tau)^{-\frac{1}{2}-\frac{N-\lambda}{2 q_{1}}}\|\tilde{u} \otimes \tilde{c}\|_{q_{1}, \lambda}
+(t-\tau)^{-\frac{N-\lambda}{ q_{2}}}\||\nabla\tilde{c}|^{2}\|_{\frac{q_{2}}{2}, \lambda} \mathrm{d} \tau\nonumber\\
&&\lesssim
\left\{\left[\sup _{t>0}\left(t^{\frac{\alpha_{1}}{2}}\|\tilde{u}(\cdot, t)\|_{q_{1}, \lambda}\right)\right]\|\tilde{c}(\cdot, t)\|_{L^{\infty}}
+\left[\sup _{t>0}\left(t^{\frac{\alpha_{2}}{2}}\|\nabla\tilde{c}(\cdot, t)\|_{q_{2}, \lambda}\right)\right]^{2}\right\}\nonumber\\
&&\quad\times \int_{0}^{t}(t-\tau)^{-\frac{1}{2}-\frac{N-\lambda}{2 q_{1}}} \tau^{-\frac{\alpha_{1}}{2}}+(t-\tau)^{-\frac{N-\lambda}{ q_{2}}} \tau^{-\alpha_{2}} \mathrm{d} \tau\nonumber\\
&&\lesssim
\|\tilde{u}\|_{\mathbb{X}}\|\tilde{c}\|_{\mathbb{Y}}
  +\|\tilde{c}\|_{\mathbb{Y}}^{2}.
\end{eqnarray}

Moreover, note that
$r_{1}>N-\lambda,  r_{2}>N-\lambda$
by Lemmas \ref{lem1}, \ref{lem3} and \ref{lemy23},  we have
\begin{eqnarray}\label{313}
\left\|e^{t \Delta} u_{0}\right\|_{\mathbb{X}}
 &:=&\sup _{t>0}\left\|e^{t \Delta} u_{0}\right\|_{\dot{N}_{r_{1}, \lambda, \infty}^{-\beta_{1}}}+\sup _{t>0}\left[t^{\frac{\alpha_{1}}{2}}\left\|e^{t \Delta} u_{0}\right\|_{M_{q_{1}, \lambda}}\right]\nonumber\\
&\lesssim&\left\|u_{0}\right\|_{\dot{\mathbf{N}}_{r_{1}, \lambda, \infty}^{-\beta_{1}}}+\left\|u_{0}\right\|_{\dot{\mathbf{N}}_{q_{1}, \lambda, \infty}^{-\alpha_{1}}} \nonumber\\
&\lesssim&\left\|u_{0}\right\|_{\dot{N}_{r_{1}, \lambda, \infty}}^{-\beta_{1}}
\end{eqnarray}
and
\begin{eqnarray}\label{315}
\left\|e^{t \Delta} c_{0}\right\|_{\mathbb{Y}}
 &:=&\sup _{t>0}\left\|\nabla e^{t \Delta} c_{0}\right\|_{\dot{\mathbf{N}}_{r_{2}, \lambda, \infty}^{-\beta_{2}}}+\sup _{t>0}\left[t^{\frac{\alpha_{2}}{2}}\left\|\nabla e^{t \Delta} c_{0}\right\| _{M_{q_{2}, \lambda}}\right]+\sup _{t>0}\left\|e^{t \Delta} c_{0}\right\|_{L^{\infty}}\nonumber\\
 &\lesssim&\left\|\nabla c_{0}\right\|_{\dot{N}_{r_{2}, \lambda, \infty}^{-\beta_{2}}}+\left\|\nabla c_{0}\right\|_{\dot{N}_{q_{2}, \lambda, \infty}^{-\alpha_{2}}}+\left\|c_{0}\right\|_{L^{\infty}}\nonumber\\
 &\lesssim&\left\|\nabla c_{0}\right\|_{\dot{\mathbf{N}}_{r_{2}, \lambda, \infty}^{-\beta_{2}}}+\left\|c_{0}\right\|_{L^{\infty}}.
\end{eqnarray}

Combining (\ref{306})-(\ref{315}) and (\ref{302}), we obtain
\begin{eqnarray}
&&\|(u,  c)\|_{\Theta}
:=\|\mathbb{T}(\tilde{u},  \tilde{c})\|_{\Theta}\nonumber\\
&&\lesssim
\|\tilde{u}\|_{\mathrm{X}}^{2}+\|\tilde{c}\|_{\mathrm{Y}}^{2}
+\|\tilde{u}\|_{\mathrm{X}}\|\tilde{c}\|_{\mathrm{Y}}
+\left\|\left(u_{0}, \nabla c_{0}\right)\right\|_{\mathbb{E}}
+\left\|c_{0}\right\|_{L^{\infty}}\nonumber\\
&&\lesssim \epsilon_{0},
\end{eqnarray}
which implies that
\begin{equation*}
(u,  c)=\mathbb{T}(\tilde{u},  \tilde{c}) \in \Theta_{\epsilon_{0}}.
\end{equation*}
We thus complete the proof of Lemma \ref{lem301}.
\end{proof}
To complete the proof of Theorem \ref{thm1}, we need the following Lemma.
\begin{lem}\label{lem302}
For $\epsilon_{0}>0$ small enough, let $(\overline{u},  \overline{c}) \in \Theta_{\epsilon_{0}}$ and $(\tilde{u},  \tilde{c}) \in \Theta_{\epsilon_{0}}$ with $\left.(\overline{u},  \overline{c})\right|_{t=0}=\left.(\tilde{u},  \tilde{c})\right|_{t=0}=\left(u_{0},  c_{0}\right),$
where $\left(u_{0},  c_{0} \right)$ satisfies (\ref{106}), then the map
$\mathbb{T}=\left(\mathbb{T}_{1}, \mathbb{T}_{2}\right)$
defined in (\ref{301}) is  contractive.
\end{lem}
\begin{proof}
For simplicity, we write $\left(u^{*},  c^{*}\right)=(\tilde{u}-\overline{u}, \tilde{c}-\overline{c}) .$ Then we have
\begin{eqnarray*}
\left|\mathbb{T}_{1}(\tilde{u},  \tilde{c})-\mathbb{T}_{1}(\overline{u}, \overline{c})\right|=\left|\int_{0}^{t} e^{(t-s) \Delta} \mathbb{P}\left(\tilde{u} \cdot \nabla u^{*}+u^{*} \cdot \nabla \overline{u}+
\operatorname{div}(\nabla c^{*} \odot \nabla \overline{c}
+\nabla \tilde{c} \odot \nabla c^{*})
\right)(\cdot, s) \mathrm{d} s\right|
\end{eqnarray*}
and
\begin{eqnarray*}
&&\left|\mathbb{T}_{2}(\tilde{u},  \tilde{c})-\mathbb{T}_{2}(\overline{u}, \overline{c})\right|\nonumber\\
&&=\left|\int_{0}^{t} e^{(t-s) \Delta}\left(-\tilde{u} \cdot \nabla c^{*}
-u^{*} \cdot \nabla \overline{c}+
(|\nabla \tilde{c}|^{2}-|\nabla \overline{c}|^{2})(\tilde{c}-\underline{d}_{0})
+|\nabla \overline{c}|^{2}c^{*}
\right)(\cdot, s) \mathrm{d} s\right|.
\end{eqnarray*}

We first compute the $\dot{\mathbf{N}}_{r_{2}, \lambda, \infty}^{-\beta_{2}}$-norm of
\begin{eqnarray*}
 \nabla\int_{0}^{t} e^{(t-\tau) \Delta}\left[
(|\nabla \tilde{c}|^{2}-|\nabla \overline{c}|^{2})(\tilde{c}-\underline{d}_{0})
+|\nabla \overline{c}|^{2}c^{*}\right](\tau) \mathrm{d}\tau.
\end{eqnarray*}

Using Lemma \ref{lemy23}, one has
\begin{eqnarray}
&&\left\|\nabla \int_{0}^{t} e^{(t-\tau) \Delta}\left[(|\nabla \tilde{c}|^{2}-|\nabla \overline{c}|^{2})(\tilde{c}-\underline{d}_{0})
+|\nabla \overline{c}|^{2}c^{*}\right](\tau) \mathrm{d}\tau\right\|_{\dot{\mathbf{N}}_{r_{2}, \lambda, \infty}^{-\beta_{2}}}\\
&&=\sup _{s>0}\left[s^{\beta_{2} / 2} \left\| \nabla e^{s \Delta} \int_{0}^{t} e^{(t-\tau) \Delta}\left[(|\nabla \tilde{c}|^{2}-|\nabla \overline{c}|^{2})(\tilde{c}-\underline{d}_{0})
+|\nabla \overline{c}|^{2}c^{*}\right](\cdot, \tau) \mathrm{d} \tau\right\|_{r_{2}, \lambda}\right]\nonumber\\
&&\leq \int_{0}^{t} \sup _{s>0}\left[s^{\beta_{2} / 2}\left\|e^{s \Delta} \nabla e^{(t-\tau) \Delta}\left[(|\nabla \tilde{c}|^{2}-|\nabla \overline{c}|^{2})(\tilde{c}-\underline{d}_{0})
+|\nabla \overline{c}|^{2}c^{*}\right](\cdot, \tau)\right\|_{r_{2}, \lambda}\right] \mathrm{d} \tau.\nonumber
\end{eqnarray}

Employing
 the fact that
 $|\tilde{c}-\underline{d}_{0}|\leq 3$ and
  $||\nabla \tilde{c}|^{2}-|\nabla \overline{c}|^{2}|
  \leq C|\nabla c^{*}|(|\nabla \tilde{c}|+|\nabla \overline{c}|)$,
one obtains by Lemmas \ref{lem1} and \ref{lem3} that
\begin{eqnarray*}\label{my1}
&&\sup _{0<s \leq t-\tau}\left[s^{\frac{\beta_{2}}{2}}\left\|e^{s \Delta} \nabla e^{(t-\tau) \Delta}\left[(|\nabla \tilde{c}|^{2}-|\nabla \overline{c}|^{2})(\tilde{c}-\underline{d}_{0})
+|\nabla \overline{c}|^{2}c^{*}\right](\cdot, \tau)\right\|_{r_{2}, \lambda}\right]\nonumber\\
&&\lesssim
(t-\tau)^{\frac{\beta_{2}}{2}}(t-\tau)^{-\frac{1}{2}-\frac{N-\lambda}{2}\left(
\frac{2}{q_{2}}-\frac{1}{r_{2}}\right)}\||\nabla c^{*}|(|\nabla \tilde{c}|+|\nabla \overline{c}|)\|_{\frac{q_{2}}{2}, \lambda}\nonumber\\
&&\quad+(t-\tau)^{\frac{\beta_{2}}{2}}(t-\tau)^{-\frac{1}{2}-\frac{N-\lambda}{2}
\left(\frac{2}{q_{2}}-\frac{1}{r_{2}}\right)}\| |\nabla \overline{c}|^{2}c^{*} \|_{\frac{q_{2}}{2}, \lambda}\nonumber\\
&&\lesssim
(t-\tau)^{-\frac{N-\lambda}{ q_{2}}}\|\nabla c^{*}\|_{q_{2}, \lambda}\|(\nabla \tilde{c}, \nabla \overline{c})\|_{q_{2}, \lambda}
+(t-\tau)^{-\frac{N-\lambda}{ q_{2}}}\|\nabla \overline{c}(\cdot, \tau)\|_{q_{2}, \lambda}^{2}\| c^{*}\|_{L^{\infty}}.
\end{eqnarray*}
For $s>t-\tau,$ note that $(t-\tau+s) / 2<s<t-\tau+s,$
it then holds that
\begin{eqnarray*}\label{my2}
&&\sup _{s > t-\tau}\left[s^{\frac{\beta_{2}}{2}}\left\|e^{s \Delta} \nabla e^{(t-\tau) \Delta}\left[(|\nabla \tilde{c}|^{2}-|\nabla \overline{c}|^{2})(\tilde{c}-\underline{d}_{0})
+|\nabla \overline{c}|^{2}c^{*}\right](\cdot, \tau)\right\|_{r_{2}, \lambda}\right]\nonumber\\
&&\lesssim
s^{\frac{\beta_{2}}{2}}(s+t-\tau)^{-\frac{1}{2}-\frac{N-\lambda}{2}
\left(\frac{2}{q_{2}}-\frac{1}{r_{2}}\right)}\||\nabla c^{*}|(|\nabla \tilde{c}|+|\nabla \overline{c}|)\|_{\frac{q_{2}}{2}, \lambda}\nonumber\\
&&\quad+s^{\frac{\beta_{2}}{2}}(s+t-\tau)^{-\frac{1}{2}-\frac{N-\lambda}{2}
\left(\frac{2}{q_{2}}-\frac{1}{r_{2}}\right)}\| |\nabla \overline{c}|^{2}c^{*} \|_{\frac{q_{2}}{2}, \lambda}\nonumber\\
&&\lesssim
(t-\tau)^{-\frac{N-\lambda}{ q_{2}}}
\sup _{s>t-\tau}\left(1+\frac{s}{t-\tau}\right)^{-\frac{N-\lambda}{ q_{2}}}\|\nabla c^{*}\|_{q_{2}, \lambda}\|(\nabla \tilde{c}, \nabla \overline{c})\|_{q_{2}, \lambda}\nonumber\\
&&\quad+(t-\tau)^{-\frac{N-\lambda}{ q_{2}}}
\sup _{s>t-\tau}\left(1+\frac{s}{t-\tau}\right)^{-\frac{N-\lambda}{q_{2}}}
\|\nabla \overline{c}(\cdot, \tau)\|_{q_{2}, \lambda}^{2}\| c^{*}\|_{L^{\infty}}.
\end{eqnarray*}
Since $ \frac{1}{q_{2}}<\frac{1}{N-\lambda},$ we thus obtain
\begin{eqnarray*}\label{my3}
&&\left\|\nabla \int_{0}^{t} e^{(t-\tau) \Delta}\left[(|\nabla \tilde{c}|^{2}-|\nabla \overline{c}|^{2})(\tilde{c}-\underline{d}_{0})
+|\nabla \overline{c}|^{2}c^{*}\right](\tau) \mathrm{d}\tau\right\|_{\dot{\mathbf{N}}_{r_{2}, \lambda, \infty}^{-\beta_{2}}}\nonumber\\
&&\lesssim \bigg\{\left[\sup _{t>0}\left(t^{\frac{\alpha_{2}}{2}}\|\nabla c^{*}(\cdot, t)\|_{q_{2}, \lambda}\right)\right]\left[\sup _{t>0}\left(t^{\frac{\alpha_{2}}{2}}\left(\|\nabla \tilde{c}(\cdot, t)\|_{q_{2}, \lambda}+\|\nabla \overline{c}(\cdot, t)\|_{q_{2}, \lambda}\right)\right)\right]\nonumber\\
&&\qquad+\left[\sup _{t>0}\left(t^{\alpha_{2}}\|\nabla \overline{c}(\cdot, t)\|_{q_{2}, \lambda}^{2}\right)\right] \|c^{*}\|_{L\infty}\bigg\}
\times \int_{0}^{t}(t-\tau)^{-\frac{N-\lambda}{ q_{2}}} \tau^{-\alpha_{2}} \mathrm{d} \tau\nonumber\\
&&\lesssim \|c^{*}\|_{\mathbb{Y}}(\|\tilde{c}\|_{\mathbb{Y}}
+\|\overline{c}\|_{\mathbb{Y}}
+\|\overline{c}\|_{\mathbb{Y}}^{2}).
\end{eqnarray*}
Next, we calculate the $M_{q_{2}, \lambda}$-norm of
$\nabla\int_{0}^{t} e^{(t-\tau) \Delta}\left[(|\nabla \tilde{c}|^{2}-|\nabla \overline{c}|^{2})(\tilde{c}-\underline{d}_{0})
+|\nabla \overline{c}|^{2}c^{*}\right](\tau) \mathrm{d}\tau$.
\begin{eqnarray*}\label{my4}
  &&\left\|\nabla \int_{0}^{t} e^{(t-\tau) \Delta}\left[(|\nabla \tilde{c}|^{2}-|\nabla \overline{c}|^{2})(\tilde{c}-\underline{d}_{0})
+|\nabla \overline{c}|^{2}c^{*}\right](\cdot, \tau) \mathrm{d} \tau\right\|_{q_{2}, \lambda}\nonumber\\
  &&\lesssim
  \int_{0}^{t}(t-\tau)^{-\frac{1}{2}-\frac{N-\lambda}{2}
  \left(\frac{2}{q_{2}}-\frac{1}{q_{2}}\right)}
  \|\nabla c^{*}\|_{q_{2}, \lambda}\|(\nabla \tilde{c}, \nabla \overline{c})\|_{q_{2}, \lambda}\nonumber\\
  &&\qquad\qquad+(t-\tau)^{-\frac{1}{2}-\frac{N-\lambda}{2}
  \left(\frac{2}{q_{2}}-\frac{1}{q_{2}}\right)} \| |\nabla \overline{c}|^{2}c^{*}\|_{\frac{ q_{2}}{2}, \lambda}\mbox{d}\tau\nonumber\\
  &&\lesssim
  \bigg\{\left[\sup _{t>0}\left(t^{\frac{\alpha_{2}}{2}}\|\nabla c^{*}\|_{q_{2}, \lambda}\right)\right]\left[\sup _{t>0}\left(t^{\frac{\alpha_{2}}{2}}(\|\nabla \tilde{c}\|_{q_{2}, \lambda}
  +\|\nabla \overline{c}\|_{q_{2}, \lambda})\right)\right]\nonumber\\
  &&\qquad+\left[\sup _{t>0}\left(t^{\alpha_{2}}\|\nabla \overline{c}\|_{q_{2}, \lambda}^{2}\right)\right]\|c^{*}\|_{L^{\infty}}\bigg\}
  \times \int_{0}^{t}
  (t-\tau)^{-\frac{1}{2}-\frac{N-\lambda}{2}\left(\frac{2}{q_{2}}-\frac{1}{q_{2}}\right)} \tau^{-\alpha_{2}} \mbox{d}\tau\nonumber\\
  &&\lesssim t^{-\frac{1}{2}+\frac{N-\lambda}{2 q_{2}}}\|c^{*}\|_{\mathbb{Y}}(\|\tilde{c}\|_{\mathbb{Y}}
+\|\overline{c}\|_{\mathbb{Y}}
+\|\overline{c}\|_{\mathbb{Y}}^{2}).
\end{eqnarray*}

We calculate $L^{\infty}$-norm of
$\int_{0}^{t} e^{(t-\tau) \Delta}\left[(|\nabla \tilde{c}|^{2}-|\nabla \overline{c}|^{2})(\tilde{c}-\underline{d}_{0})
+|\nabla \overline{c}|^{2}c^{*}\right](\tau) \mathrm{d}\tau$.
Note that $q_{2}>N-\lambda,$ so there holds
\begin{eqnarray*}\label{my5}
&&\left\|\int_{0}^{t} e^{(t-\tau) \Delta}\left[(|\nabla \tilde{c}|^{2}-|\nabla \overline{c}|^{2})(\tilde{c}-\underline{d}_{0})
+|\nabla \overline{c}|^{2}c^{*}\right](\cdot, \tau) \mathrm{d} \tau\right\|_{L^{\infty}}\nonumber\\
&&\lesssim
\int_{0}^{t}(t-\tau)^{-\frac{N-\lambda}{q_{2}}}\|\nabla c^{*}\|_{q_{2}, \lambda}
\|(\nabla \tilde{c},\nabla \overline{c})\|_{q_{2}, \lambda}+(t-\tau)^{-\frac{N-\lambda}{ q_{2}}}\|\nabla\overline{c}\|_{q_{2}, \lambda}^{2}\|c^{*}\|_{L^{\infty}}\mathrm{d} \tau\nonumber\\
&&\lesssim
\bigg\{\left[\sup _{t>0}\left(t^{\frac{\alpha_{2}}{2}}\|\nabla c^{*}(\cdot, t)\|_{q_{2}, \lambda}\right)\right]\left[\sup _{t>0}\left(t^{\frac{\alpha_{2}}{2}}(\|\nabla \tilde{c}\|_{q_{2}, \lambda}
  +\|\nabla \overline{c}\|_{q_{2}, \lambda})\right)\right]\nonumber\\
&&\qquad+\left[\sup _{t>0}\left(t^{\frac{\alpha_{2}}{2}}\|\nabla\overline{c}(\cdot, t)\|_{q_{2}, \lambda}\right)\right]^{2}\|c^{*}\|_{L\infty}\bigg\}
\times \int_{0}^{t}(t-\tau)^{-\frac{N-\lambda}{ q_{2}}} \tau^{-\alpha_{2}} \mathrm{d} \tau\nonumber\\
&&\lesssim
\|c^{*}\|_{\mathbb{Y}}(\|\tilde{c}\|_{\mathbb{Y}}
+\|\overline{c}\|_{\mathbb{Y}}
+\|\overline{c}\|_{\mathbb{Y}}^{2}).
\end{eqnarray*}

For the estimates of the  remaining part of $\left\|\mathbb{T}_{1}(\tilde{u},  \tilde{c})-\mathbb{T}_{1}(\overline{u}, \overline{c})\right\|_{\mathbb{X}}$
and $\left\|\mathbb{T}_{2}(\tilde{u},  \tilde{c})-\mathbb{T}_{2}(\overline{u}, \overline{c})\right\|_{\mathbb{Y}}$,
we can repeat  the proof of the corresponding part as in Lemma \ref{lem301},
and thus  we conclude that
\begin{eqnarray*}
&&\left\|\mathbb{T}_{1}(\tilde{u},  \tilde{c})-\mathbb{T}_{1}(\overline{u}, \overline{c})\right\|_{\mathbb{X}}\nonumber\\
 &&\lesssim
\left\|u^{*}\right\|_{\mathbb{X}}\left(\|\tilde{u}\|_{\mathbb{X}}
+\|\overline{u}\|_{\mathbb{X}}\right)
+\left\| c^{*}\right\|_{\mathbb{Y}}\left(\|\tilde{c}\|_{\mathbb{Y}}
+\|\overline{c}\|_{\mathbb{Y}}\right)\nonumber\\
&&\lesssim C_{1}\epsilon_{0}(\left\|u^{*}\right\|_{\mathbb{X}}+\left\| c^{*}\right\|_{\mathbb{Y}})
\end{eqnarray*}
and
\begin{eqnarray*}
&&\left\|\mathbb{T}_{2}(\tilde{u},  \tilde{c})-\mathbb{T}_{2}(\overline{u}, \overline{c})\right\|_{\mathbb{Y}}\nonumber\\
 &&\lesssim
\left(\|u^{*}\right\|_{\mathbb{X}}
+\left\| c^{*}\right\|_{\mathbb{Y}})(\|\tilde{u}\|_{\mathbb{X}}
+\|\overline{u}\|_{\mathbb{X}}+\|\tilde{c}\|_{\mathbb{Y}}
+\|\overline{c}\|_{\mathbb{Y}}+\|\overline{c}\|_{\mathbb{Y}}^{2})\nonumber\\
&&\lesssim
C_{2}\epsilon_{0}(\left\|u^{*}\right\|_{\mathbb{X}}+\left\| c^{*}\right\|_{\mathbb{Y}}).
\end{eqnarray*}
Choosing  $\epsilon_{0}>0$ small enough so that $\left(C_{1}+C_{2}\right) \epsilon_{0} \leq \frac{1}{2},$ we can then prove Lemma \ref{lem302}.
\end{proof}

\section{Proof of Theorem \ref{thm2}--large time behavior}
In this section we prove Theorem \ref{thm2}.

The proof of Theorem \ref{thm2} is a consequence of the
following Lemma \ref{lem401}.

 Let $(\overline{u}, \overline{c})$ and $(\tilde{u},  \tilde{c}),$ respectively, be the solutions
of $(\ref{302})$ constructed in Theorem \ref{thm1} corresponding to the initial data $\left(\overline{u}_{0},  \overline{c}_{0}\right)$ and $\left(\tilde{u}_{0},  \tilde{c}_{0}\right),$
respectively.
According to  Theorem \ref{thm1}, there exists a constant $C_{0}$ such that
\begin{equation}\label{401}
\|(\overline{u}, \overline{c})\|_{\Theta} \leq C_{0} \epsilon_{0}, \quad\|(\tilde{u}, \tilde{c})\|_{\Theta} \leq C_{0} \epsilon_{0}.
\end{equation}
Let $\left(u^{*}, c^{*}\right)=(\tilde{u}-\overline{u}, \tilde{c}-\overline{c}),$ then we have
\begin{eqnarray*}
&&\tilde{u}-\overline{u}
=e^{t \Delta}(\tilde{u}_{0}-\overline{u}_{0})\nonumber\\
&&\quad\quad\quad\quad+\int_{0}^{t} e^{(t-s) \Delta} \mathbb{P}\left(\tilde{u} \cdot \nabla u^{*}+u^{*} \cdot \nabla \overline{u}+
\operatorname{div}(\nabla c^{*} \odot \nabla \overline{c}
+\nabla \tilde{c} \odot \nabla c^{*})
\right)(\cdot, s) \mathrm{d} s,\\
&&\tilde{c}-\overline{c}
=e^{t \Delta}(\tilde{c}_{0}-\overline{c}_{0})\nonumber\\
&&\quad\quad\quad\quad+
\int_{0}^{t} e^{(t-s) \Delta}\left(-\tilde{u} \cdot \nabla c^{*}
-u^{*} \cdot \nabla \overline{c}+
(|\nabla \tilde{c}|^{2}-|\nabla \overline{c}|^{2})(\tilde{c}-\underline{d}_{0})
+|\nabla \overline{c}|^{2}c^{*}
\right)(\cdot, s) \mathrm{d} s.\nonumber
\end{eqnarray*}
Next, we introduce two auxiliary functions
\begin{eqnarray*}
&&h(t)
=t^{\frac{\alpha_{1}}{2}}\left\|e^{t \Delta}\left(\tilde{u}_{0}-\overline{u}_{0}\right)\right\|_{q_{1}, \lambda}+t^{\frac{\alpha_{2}}{2}}\left\|\nabla e^{t \Delta}\left(\tilde{c}_{0}-\overline{c}_{0}\right)\right\|_{q_{2}, \lambda}\nonumber\\
&&\quad\quad\quad+\left\|e^{t \Delta}\left(\tilde{u}_{0}-\overline{u}_{0}, \nabla \tilde{c}_{0}-\nabla \overline{c}_{0}\right)\right\|_{\mathbb{E}}
+\left\|e^{t \Delta}\left(\tilde{c}_{0}-\overline{c}_{0}\right)\right\|_{L^{\infty}}
\end{eqnarray*}
and
\begin{eqnarray*}
l(t)=t^{\frac{\alpha_{1}}{2}}\|\tilde{u}-\overline{u}\|_{q_{1}, \lambda}+t^{\frac{\alpha_{2}}{2}}\|\nabla(\tilde{c}-\overline{c})\|_{q_{2}, \lambda}
+\|(\tilde{u}-\overline{u},  \nabla \tilde{c}-\nabla \overline{c})\|_{\mathbb{E}}
+\|\tilde{c}-\overline{c}\|_{L^{\infty}}.
\end{eqnarray*}

\begin{lem}\label{lem401}
There holds
\begin{equation}\label{403}
\lim _{t \rightarrow \infty} h(t)=0
\end{equation}
$$\Updownarrow$$
\begin{equation}\label{404}
\lim _{t \rightarrow \infty} l(t)=0.
\end{equation}
\end{lem}
\begin{proof}
We just prove
$"(\ref{403})\Longrightarrow (\ref{404})",$ the proof of the opposite direction is similar.

First,
%, we calculate the
%\begin{equation*}
%\dot{\mathbf{N}}_{r_{1}, \lambda, \infty}^{-\beta_{1}} \times \dot{\mathbf{N}}_{r_{2}, \lambda, \infty}^{-\beta_{2}}
%\end{equation*}
%of
%\begin{equation*}
%(\overline{u}-\tilde{u}, \nabla \overline{c}-\nabla \tilde{c}).
%\end{equation*}
invoking  Lemma \ref{lemy23} and the boundedness of $\mathbb{P}$ in Morrey space,
one finds
\begin{eqnarray}
&&\|(\overline{u}-\tilde{u},  \nabla \overline{c}-\nabla \tilde{c})\|_{\dot{\mathbf{N}}_{r_{1}, \lambda, \infty}^{-\beta_{1}} \times \dot{\mathbf{N}}_{r_{2}, \lambda, \infty}^{-\beta_{2}}}\nonumber\\
&&=\left\|\left(e^{t \Delta} \overline{u}_{0}-e^{t \Delta} \tilde{u}_{0}, \nabla e^{t \Delta} \overline{c}_{0}-\nabla e^{t \Delta} \tilde{c}_{0}\right)\right\|_{\dot{\mathbf{N}}_{r_{1}, \lambda, \infty}^{-\beta_{1}} \times \dot{\mathbf{N}}_{r_{2}, \lambda, \infty}^{-\beta_{2}} }\nonumber\\
&&\quad+\int_{0}^{t} \sup _{s>0}\bigg[s^{\beta_{1} / 2}\big\|e^{s \Delta} e^{(t-\tau) \Delta} \mathbb{P}(\tilde{u} \cdot \nabla u^{*}+u^{*} \cdot \nabla \overline{u}\nonumber\\
&&\quad\quad+
\operatorname{div}(\nabla c^{*} \odot \nabla \overline{c}
+\nabla \tilde{c} \odot \nabla c^{*}))(\cdot, \tau)\big\|_{r_{1}, \lambda}\bigg]\mathrm{d} \tau\nonumber\\
&&\quad+\int_{0}^{t} \sup _{s>0}\bigg[s^{\beta_{2} / 2}\big\|e^{s \Delta} \nabla e^{(t-\tau) \Delta} \big(\tilde{u} \cdot \nabla c^{*}
+u^{*} \cdot \nabla \overline{c}\nonumber\\
&&\quad\quad+
(|\nabla \tilde{c}|^{2}-|\nabla \overline{c}|^{2})(\tilde{c}-\underline{d}_{0})
+|\nabla \overline{c}|^{2}c^{*}\big)(\cdot, \tau)\big\|_{r_{2}, \lambda}\bigg] \mathrm{d} \tau\nonumber\\
&&=\left\|\left(e^{t \Delta} \overline{u}_{0}-e^{t \Delta} \tilde{u}_{0}, \nabla e^{t \Delta} \overline{c}_{0}-\nabla e^{t \Delta} \tilde{c}_{0}\right)\right\|_{\mathbb{E}}+I+II.
\end{eqnarray}

Employing Lemmas \ref{lem1} and \ref{lem3}, we have
\begin{eqnarray}
&&\sup _{s>0}\left[s^{\frac{\beta_{1}}{2}}\left\|e^{s \Delta} e^{(t-\tau) \Delta} \mathbb{P}\left(\tilde{u} \cdot \nabla u^{*}+u^{*} \cdot \nabla \overline{u}
+\operatorname{div}(\nabla \tilde{c} \odot \nabla c^{*}
+\nabla c^{*} \odot \nabla \overline{c})
\right)(\cdot, \tau)\right\|_{r_{1}, \lambda}\right]\nonumber\\
&&\leq
(t-\tau)^{-\frac{N-\lambda}{q_{1}}}\left\|u^{*}(\cdot, \tau)\right\|_{q_{1}, \lambda}\|(\tilde{u}(\cdot, \tau), \overline{u}(\cdot, \tau))\|_{q_{1}, \lambda}\nonumber\\
&&\quad+(t-\tau)^{-\frac{N-\lambda}{q_{2}}}\left\|\nabla c^{*}(\cdot, \tau)\right\|_{q_{2}, \lambda}\|(\nabla\tilde{c}(\cdot, \tau), \nabla\overline{c}(\cdot, \tau))\|_{q_{2}, \lambda}
\end{eqnarray}
and
\begin{eqnarray}
&&\sup _{s>0}\left[s^{\frac{\beta_{2}}{2}}\left\|e^{s \Delta} \nabla e^{(t-\tau) \Delta}\left(\tilde{u} \cdot \nabla c^{*}+u^{*} \cdot \nabla \overline{c}
+(|\nabla \tilde{c}|^{2}-|\nabla \overline{c}|^{2})(\tilde{c}-\underline{d}_{0})
+|\nabla \overline{c}|^{2}c^{*}\right)(\cdot, \tau)\right\|_{r_{2}, \lambda}\right]\nonumber\\
&&\leq
(t-\tau)^{-\frac{N-\lambda}{2 q_{1}}-\frac{N-\lambda}{2 q_{2}}}\left(\|\tilde{u}(\cdot, \tau)\|_{q_{1}, \lambda}\left\|\nabla c^{*}(\cdot, \tau)\right\|_{q_{2}, \lambda}+\left\|u^{*}(\cdot, \tau)\right\|_{q_{1}, \lambda}\|\nabla \overline{c}\|_{q_{2}, \lambda}\right)\nonumber\\
&&\quad+(t-\tau)^{-\frac{N-\lambda}{q_{2}}}
\left\|\nabla c^{*}(\cdot, \tau)\right\|_{q_{2}, \lambda}
(\|\nabla \overline{c}\|_{q_{2}, \lambda}
+\|\nabla \tilde{c}\|_{q_{2}, \lambda}).
\end{eqnarray}

Let $0 < \delta<1$, using (\ref{401}),
  we estimate $I$ as follows
\begin{eqnarray}\label{405}
&&I
\lesssim
\left(\int_{0}^{\delta t}+\int_{\delta t}^{t}\right)\big[(t-\tau)^{-\frac{N-\lambda}{q_{1}}}\left\|u^{*}\right\|_{q_{1}, \lambda}\|(\tilde{u}, \overline{u})\|_{q_{1}, \lambda}\nonumber\\
&&\qquad\qquad\qquad\qquad+(t-\tau)^{-\frac{N-\lambda}{q_{2}}}\left\|\nabla c^{*}\right\|_{q_{2}, \lambda}\|(\nabla\tilde{c}, \nabla\overline{c})\|_{q_{2}, \lambda}\big]\mbox{d}\tau\nonumber\\
&&\lesssim
\|(\overline{u}, \tilde{u})\|_{\mathbb{X}} \int_{0}^{\delta t}(t-\tau)^{-\frac{N-\lambda}{q_{1}}} \tau^{-\alpha_{1}}\left(\tau^{\frac{\alpha_{1}}{2}}\left\|u^{*}(\cdot, \tau)\right\|_{q_{1}, \lambda}\right) \mathrm{d} \tau\nonumber\\
&&\quad+\|(\overline{c}, \tilde{c})\|_{\mathbb{Y}} \int_{0}^{\delta t}(t-\tau)^{-\frac{N-\lambda}{q_{2}}} \tau^{-\alpha_{2}}\left(\tau^{\frac{\alpha_{2}}{2}}\left\|\nabla c^{*}(\cdot, \tau)\right\|_{q_{2}, \lambda}\right) \mathrm{d} \tau\nonumber\\
&&\quad+\|(\overline{u}, \tilde{u})\|_{\mathbb{X}}\left[\sup _{\delta t \leq \tau \leq t}\left(\tau^{\frac{\alpha_{1}}{2}}\left\|u^{*}(\cdot, \tau)\right\|_{q_{1}, \lambda}\right)\right]\nonumber\\
&&\quad+\|(\overline{c}, \tilde{c})\|_{\mathbb{X}}\left[\sup _{\delta t \leq \tau \leq t}\left(\tau^{\frac{\alpha_{2}}{2}}\left\|\nabla c^{*}(\cdot, \tau)\right\|_{q_{2}, \lambda}\right)\right]\nonumber\\
&&\lesssim \epsilon_{0} \int_{0}^{\delta t}(t-\tau)^{-\frac{N-\lambda}{q_{1}}} \tau^{-\alpha_{1}}\left(\tau^{\frac{\alpha_{1}}{2}}\left\|u^{*}(\cdot, \tau)\right\|_{q_{1}, \lambda}\right) \mathrm{d} \tau\nonumber\\
&&\quad+\epsilon_{0} \int_{0}^{\delta t}(t-\tau)^{-\frac{N-\lambda}{q_{2}}} \tau^{-\alpha_{2}}\left(\tau^{\frac{\alpha_{2}}{2}}\left\|\nabla c^{*}(\cdot, \tau)\right\|_{q_{2}, \lambda}\right) \mathrm{d}\tau\nonumber\\
&&\quad+\epsilon_{0} \left\{\sup _{\delta t \leq \tau \leq t}\left[\tau^{\frac{\alpha_{1}}{2}}\left\|u^{*}(\cdot, \tau)\right\|_{q_{1}, \lambda}\right]+
\sup _{\delta t \leq \tau \leq t}\left[\tau^{\frac{\alpha_{2}}{2}}\left\|\nabla c^{*}(\cdot, \tau)\right\|_{q_{2}, \lambda}\right]\right\}.
\end{eqnarray}

For $II$ , we use the same argument as above to get
\begin{eqnarray}\label{407}
&&II \lesssim
\epsilon_{0} \int_{0}^{\delta t}(t-\tau)^{-\frac{N-\lambda}{2 q_{1}}-\frac{N-\lambda}{2 q_{2}}} \tau^{-\frac{\alpha_{1}+\alpha_{2}}{2}}\left(\tau^{\frac{\alpha_{2}}{2}}\left\|\nabla c^{*}(\cdot, \tau)\right\|_{q_{2}, \lambda}+\tau^{\frac{\alpha_{1}}{2}}\left\|u^{*}(\cdot, \tau)\right\|_{q_{1}, \lambda}\right) \mathrm{d}\tau\nonumber\\
&&\qquad+\epsilon_{0} \int_{0}^{\delta t}(t-\tau)^{-\frac{N-\lambda}{ q_{2}}} \tau^{-\alpha_{2}}\left(\tau^{\frac{\alpha_{2}}{2}}\left\|\nabla c^{*}(\cdot, \tau)\right\|_{q_{2}, \lambda}\right) \mathrm{d} \tau\nonumber\\
&&\qquad+\epsilon_{0}\left[\sup _{\delta t \leq \tau \leq t}\left(\tau^{\alpha_{1} / 2}\left\|u^{*}(\cdot, \tau)\right\|_{q_{1}, \lambda}\right)
+\sup _{\delta t \leq \tau \leq t}\left(\tau^{\alpha_{2} / 2}\left\|\nabla c^{*}(\cdot, \tau)\right\|_{q_{2}, \lambda}\right)\right].
\end{eqnarray}
From (\ref{405})-(\ref{407}), we get
\begin{eqnarray}\label{408}
&&\|(\overline{u}-\tilde{u},  \nabla \overline{c}-\nabla \tilde{c})\|_{\dot{\mathbf{N}}_{r_{1}, \lambda, \infty}^{-\beta_{1}}
\times \dot{\mathbf{N}}_{r_{2}, \lambda, \infty}^{-\beta_{2}}
} \nonumber\\
&&=\left\|\left(e^{t \Delta} \overline{u}_{0}-e^{t \Delta} \tilde{u}_{0},  \nabla e^{t \Delta} \overline{c}_{0}-\nabla e^{t \Delta} \tilde{c}_{0}\right)\right\|_{\mathbb{E}}+I+II\nonumber\\
&&\lesssim \epsilon_{0} \int_{0}^{\delta}(1-s)^{-\frac{N-\lambda}{q_{1}}} s^{-\alpha_{1}}\left[(t s)^{\alpha_{1} / 2}\left\|u^{*}(t s)\right\|_{q_{1}, \lambda}\right] \mathrm{d} s\nonumber\\
&&\quad+\epsilon_{0} \int_{0}^{\delta}(1-s)^{-\frac{N-\lambda}{q_{2}}} s^{-\alpha_{2}}\left[(t s)^{\alpha_{2} / 2}\left\|\nabla c^{*}(t s)\right\|_{q_{2}, \lambda}\right] \mathrm{d} s\nonumber\\
&&\quad+\epsilon_{0} \int_{0}^{\delta}(1-s)^{-\frac{N-\lambda}{2 q_{1}}-\frac{N-\lambda}{2 q_{2}}} s^{-\frac{\alpha_{1}+\alpha_{2}}{2}}\left[(t s)^{\frac{\alpha_{2}}{2}}\left\|\nabla c^{*}(t s)\right\|_{q_{2}, \lambda}+(t s)^{\frac{\alpha_{1}}{2}}\left\|u^{*}(t s)\right\|_{q_{1}, \lambda}\right] \mathrm{d} s\nonumber\\
&&\quad+\epsilon_{0}\left[\sup _{\delta t \leq \tau \leq t}\left(\tau^{\alpha_{1} / 2}\left\|u^{*}(\cdot, \tau)\right\|_{q_{1}, \lambda}\right)
+
\sup _{\delta t \leq \tau \leq t}\left(\tau^{\alpha_{2} / 2}\left\|\nabla c^{*}(\cdot, \tau)\right\|_{q_{2}, \lambda}\right)\right].
\end{eqnarray}

Next, we calculate
 $M_{q_{1}, \lambda} \times M_{q_{2}, \lambda}  \times L^{\infty}$ norm of $(\overline{u}-\tilde{u},  \nabla \overline{c}-\nabla \tilde{c}, \overline{c}-\tilde{c})$.
\begin{eqnarray*}
&&\|(\overline{u}-\tilde{u},  \nabla \overline{c}-\nabla \tilde{c},
 \overline{c}-\tilde{c})\|_{M_{q_{1}, \lambda}
  \times M_{q_{2}, \lambda}\times L^{\infty} }\nonumber\\
&&\lesssim\left\|\left(e^{t \Delta} \overline{u}_{0}-e^{t \Delta} \tilde{u}_{0}, \nabla e^{t \Delta} \overline{c}_{0}-\nabla e^{t \Delta} \tilde{c}_{0},
e^{t \Delta} \overline{c}_{0}- e^{t \Delta} \tilde{c}_{0}
\right)\right\|_{M_{q_{1}, \lambda}
 \times M_{q_{2}, \lambda}
 \times L^{\infty}
 } \nonumber\\
&&\qquad+J_{1}+J_{2}+J_{3},
\end{eqnarray*}
where
\begin{equation*}
J_{1}=\int_{0}^{t}\left\|e^{(t-\tau) \Delta} \mathbb{P}\left(\tilde{u} \cdot \nabla u^{*}+u^{*} \cdot \nabla \overline{u}+
\operatorname{div}(\nabla \tilde{c} \odot \nabla c^{*}
+\nabla c^{*}\odot\nabla \overline{c}  )
\right)(\cdot, \tau)\right\|_{q_{1}, \lambda} \mathrm{d} \tau,
\end{equation*}

\begin{equation*}
J_{2}=\int_{0}^{t}\left\|e^{(t-\tau) \Delta} \mathbb{P}\left(\tilde{u} \cdot \nabla c^{*}+u^{*} \cdot \nabla \overline{c}
+(|\nabla \tilde{c}|^{2}-|\nabla \overline{c}|^{2})(\tilde{c}-\underline{d}_{0})
+|\nabla \overline{c}|^{2}c^{*}\right)(\cdot, \tau)\right\|_{q_{2}, \lambda}\mathrm{d} \tau,
\end{equation*}

\begin{equation*}
J_{3}=\int_{0}^{t}\left\|e^{(t-\tau) \Delta} \mathbb{P}\left(\tilde{u} \cdot \nabla c^{*}+u^{*} \cdot \nabla \overline{c}
+(|\nabla \tilde{c}|^{2}-|\nabla \overline{c}|^{2})(\tilde{c}-\underline{d}_{0})
+|\nabla \overline{c}|^{2}c^{*}\right)(\cdot, \tau)\right\|_{L^{\infty}}\mathrm{d} \tau.
\end{equation*}

 Let $0<\delta<1$, we estimate $J_{1}$ as follows
\begin{eqnarray}
&&J_{1}
\lesssim
\left(\int_{0}^{\delta t}+\int_{\delta t}^{t}\right)(t-\tau)^{-\frac{1}{2}-\frac{N-\lambda}{2}\left(\frac{2}{q_{1}}-\frac{1}{q_{1}}\right)}\left\|\tilde{u}(\cdot, \tau) \otimes u^{*}(\cdot, \tau)+u^{*}(\cdot, \tau) \otimes \overline{u}(\cdot, \tau)\right\|_{\frac{q_{1}}{2}, \lambda}\nonumber\\
&&\quad+(t-\tau)^{-\frac{1}{2}-\frac{N-\lambda}{2}\left(\frac{2}{q_{2}}-\frac{1}{q_{1}}\right)}\left\|\nabla\tilde{c}(\cdot, \tau) \otimes \nabla c^{*}(\cdot, \tau)+ \nabla c^{*}(\cdot, \tau) \otimes \nabla\overline{c}(\cdot, \tau)\right\|_{\frac{q_{2}}{2}, \lambda}\mbox{d}\tau\nonumber\\
&&\lesssim \epsilon_{0} \int_{0}^{\delta t}(t-\tau)^{-\frac{1}{2}-\frac{N-\lambda}{2}\left(\frac{2}{q_{1}}-\frac{1}{q_{1}}\right)} \tau^{-\alpha_{1}}\left(\tau^{\frac{\alpha_{1}}{2}}\left\|u^{*}(\cdot, \tau)\right\|_{q_{1}, \lambda}\right)\nonumber\\
&&\qquad\qquad+(t-\tau)^{-\frac{1}{2}-\frac{N-\lambda}{2}\left(\frac{2}{q_{2}}-\frac{1}{q_{1}}\right)} \tau^{-\alpha_{2}}\left(\tau^{\frac{\alpha_{2}}{2}}\left\|\nabla c^{*}(\cdot, \tau)\right\|_{q_{2}, \lambda}\right)\mbox{d}\tau\nonumber\\
&&\quad+\epsilon_{0} t^{-1 / 2+\frac{N-\lambda}{2 q_{1}}}\left[\sup _{\delta t \leq \tau \leq t}\left(\tau^{\frac{\alpha_{1}}{2}}\left\|u^{*}(\cdot, \tau)\right\|_{q_{1}, \lambda}\right)
+\sup _{\delta t \leq \tau \leq t}\left(\tau^{\frac{\alpha_{2}}{2}}\left\|\nabla c^{*}(\cdot, \tau)\right\|_{q_{2}, \lambda}\right)
\right].
\end{eqnarray}

Similarly,  we  estimate  $J_{2}$ as follows:
\begin{eqnarray*}
&&J_{2}
\lesssim
\left(\int_{0}^{\delta t}+\int_{\delta t}^{t}\right)(t-\tau)^{-\frac{1}{2}-\frac{N-\lambda}{2}\left(\frac{1}{q_{1}}+\frac{1}{q_{2}}-\frac{1}{q_{2}}\right)}\left\|\left(\tilde{u} \cdot \nabla c^{*}+u^{*} \cdot \nabla \overline{c}\right)(\cdot, \tau)\right\|_{\frac{q_{1} q_{2}}{q_{1}+ q_{2}}, \lambda}\nonumber\\
&&\qquad\qquad+(t-\tau)^{-\frac{1}{2}-\frac{N-\lambda}{2}\left(\frac{2}{q_{2}}-\frac{1}{q_{2}}\right)}\left\|\left(\nabla\tilde{c} \cdot \nabla c^{*}+\nabla c^{*} \cdot \nabla \overline{c}\right)(\cdot, \tau)\right\|_{ \frac{q_{2}}{2}, \lambda}\mbox{d}\tau\nonumber\\
&&\lesssim \epsilon_{0} \int_{0}^{\delta t}(t-\tau)^{-\frac{1}{2}-\frac{N-\lambda}{2}\left(\frac{1}{q_{1}}+\frac{1}{q_{2}}-\frac{1}{q_{2}}\right)} \tau^{-\frac{\alpha_{1}+\alpha_{2}}{2}}\left(\tau^{\frac{\alpha_{1}}{2}}\left\|u^{*}\right\|_{q_{1}, \lambda}+\tau^{\frac{\alpha_{2}}{2}}\left\|\nabla c^{*}(\cdot, \tau)\right\|_{q_{2}, \lambda}\right)\nonumber\\
&&\qquad\qquad+(t-\tau)^{-\frac{1}{2}-\frac{N-\lambda}{2}\left(\frac{2}{q_{2}}-\frac{1}{q_{2}}\right)} \tau^{-\alpha_{2}}\left(\tau^{\frac{\alpha_{2}}{2}}\left\|\nabla c^{*}(\cdot, \tau)\right\|_{q_{2}, \lambda}\right) \mathrm{d} \tau\nonumber\\
&&\quad+\epsilon_{0} t^{-\frac{1}{2}+\frac{N-\lambda}{2 q_{2}}}\left[\sup _{\delta t \leq \tau \leq t}\left(\tau^{\frac{\alpha_{1}}{2}}\left\|u^{*}(\cdot, \tau)\right\|_{q_{1}, \lambda}\right)+
\sup _{\delta t \leq \tau \leq t}\left(\tau^{\frac{\alpha_{2}}{2}}\left\|\nabla c^{*}(\cdot, \tau)\right\|_{q_{2}, \lambda}\right)
\right].
\end{eqnarray*}

For $J_{3}$, it holds that
\begin{eqnarray*}
&&J_{3}
\lesssim
 \epsilon_{0} \int_{0}^{\delta t}(t-\tau)^{-\frac{1}{2}-\frac{N-\lambda}{2 q_{1}}} \tau^{-\frac{\alpha_{1}}{2}}\left(\tau^{\frac{\alpha_{1}}{2}}\left\|u^{*}(\cdot, \tau)\right\|_{q_{1}, \lambda}+\left\|c^{*}(\cdot, \tau)\right\|_{L^{\infty}}\right) \mathrm{d}\tau\nonumber\\
&&\qquad\qquad+\epsilon_{0} \int_{0}^{\delta t}(t-\tau)^{-\frac{N-\lambda}{q_{2}}} \tau^{-\alpha_{2}}\left(\tau^{\frac{\alpha_{2}}{2}}\left\|\nabla c^{*}(\cdot, \tau)\right\|_{q_{2}, \lambda}+\left\|c^{*}(\cdot, \tau)\right\|_{L^{\infty}}\right) \mathrm{d} \tau\nonumber\\
&&\qquad\qquad+\epsilon_{0}\bigg[\sup _{\delta t \leq \tau \leq t}\left(\tau^{\frac{\alpha_{1}}{2}}\left\|u^{*}(\cdot, \tau)\right\|_{q_{1}, \lambda}\right)\nonumber\\
&&\qquad\qquad+\sup _{\delta t \leq \tau \leq t}\left(\tau^{\frac{\alpha_{2}}{2}}\left\|\nabla c^{*}(\cdot, \tau)\right\|_{q_{2}, \lambda}\right)
+\sup _{\delta t \leq \tau \leq t}\left\|c^{*}(\cdot, \tau)\right\|_{L^{\infty}}\bigg].
\end{eqnarray*}

Let
\begin{eqnarray*}
&&\Omega=\epsilon_{0} \int_{0}^{\delta }(1-s)^{-\frac{1}{2}-\frac{N-\lambda}{2}\left(\frac{2}{q_{1}}-\frac{1}{q_{1}}\right)} s^{-\alpha_{1}}\left((ts)^{\frac{\alpha_{1}}{2}}\left\|u^{*}(\cdot, ts)\right\|_{q_{1}, \lambda}\right)\nonumber\\
&&\qquad\qquad+(1-s)^{-\frac{1}{2}-\frac{N-\lambda}{2}\left(\frac{2}{q_{2}}-\frac{1}{q_{1}}\right)} s^{-\alpha_{2}}\left((ts)^{\frac{\alpha_{2}}{2}}\left\|\nabla c^{*}(\cdot, ts)\right\|_{q_{2}, \lambda}\right)\mbox{d}s\nonumber\\
&&\quad+\epsilon_{0} \int_{0}^{\delta }(1-s)^{-\frac{1}{2}-\frac{N-\lambda}{2q_{1}}} s^{-\frac{\alpha_{1}+\alpha_{2}}{2}}\left((ts)^{\frac{\alpha_{1}}{2}}\left\|u^{*}(ts)\right\|_{q_{1}, \lambda}+(ts)^{\frac{\alpha_{2}}{2}}\left\|\nabla c^{*}(\cdot, ts)\right\|_{q_{2}, \lambda}\right)\nonumber\\
&&\qquad\qquad+(1-s)^{-\frac{1}{2}-\frac{N-\lambda}{2}\left(\frac{2}{q_{2}}-\frac{1}{q_{2}}\right)} s^{-\alpha_{2}}\left((ts)^{\frac{\alpha_{2}}{2}}\left\|\nabla c^{*}(\cdot, ts)\right\|_{q_{2}, \lambda}\right) \mathrm{d} s\nonumber\\
&&\quad+\epsilon_{0} \int_{0}^{\delta }(1-s)^{-\frac{1}{2}-\frac{N-\lambda}{2 q_{1}}} s^{-\frac{\alpha_{1}}{2}}\left((ts)^{\frac{\alpha_{1}}{2}}\left\|u^{*}(\cdot, ts)\right\|_{q_{1}, \lambda}+\left\|c^{*}(\cdot, ts)\right\|_{L^{\infty}}\right) \mathrm{d} s\nonumber\\
&&\quad+\epsilon_{0} \int_{0}^{\delta }(1-s)^{-\frac{N-\lambda}{q_{2}}} s^{-\alpha_{2}}\left((ts)^{\frac{\alpha_{2}}{2}}\left\|\nabla c^{*}(\cdot, ts)\right\|_{q_{2}, \lambda}+\left\|c^{*}(\cdot, ts)\right\|_{L^{\infty}}\right) \mathrm{d} s.\nonumber
\end{eqnarray*}
Then we conclude that
\begin{eqnarray}\label{409}
&&\left\|(t^{\frac{\alpha_{1}}{2}} \overline{u}-t^{\frac{\alpha_{1}}{2}} \tilde{u}, t^{\frac{\alpha_{2}}{2}} \nabla \overline{c}-t^{\frac{\alpha_{2}}{2}} \nabla \tilde{c}, \overline{c}-\tilde{c})\right\|_{M_{q_{1}, \lambda} \times M_{q_{2}, \lambda}\times L^{\infty}}\nonumber\\
&&\lesssim \left\|(t^{\frac{\alpha_{1}}{2}} e^{t \Delta} \overline{u}_{0}-t^{\frac{\alpha_{1}}{2}} e^{t \Delta} \tilde{u}_{0},
t^{\frac{\alpha_{2}}{2}} \nabla e^{t \Delta} \overline{c}_{0}-t^{\frac{\alpha_{2}}{2}} \nabla e^{t \Delta} \tilde{c}_{0},
\overline{c}_{0}-\tilde{c}_{0}) \right\|_{M_{q_{1}, \lambda} \times M_{q_{2}, \lambda}\times L^{\infty} }
\nonumber\\
&&\quad+\epsilon_{0}\bigg[\sup _{\delta t \leq \tau \leq t}\left(\tau^{\frac{\alpha_{1}}{2}}\left\|u^{*}(\cdot, \tau)\right\|_{q_{1}, \lambda}\right)
+\sup _{\delta t \leq \tau \leq t}\left(\tau^{\frac{\alpha_{2}}{2}}\left\|\nabla c^{*}(\cdot, \tau)\right\|_{q_{3}, \lambda}\right)\nonumber\\
&&\qquad\qquad
+\sup _{\delta t \leq \tau \leq t}\left\|c^{*}(\cdot, \tau)\right\|_{L^{\infty}}
\bigg]+\Omega.
\end{eqnarray}
 From (\ref{313})-(\ref{315}) and  condition (\ref{403}), we have
\begin{equation}\label{410}
h(t) \in L^{\infty}([0, \infty)), \quad \lim _{t \rightarrow \infty} h(t)=0.
\end{equation}

Let
\begin{equation*}
M=\limsup _{t \rightarrow \infty} l(t)=\lim _{k \in \mathbb{Z}, k \rightarrow \infty} \sup _{t \geq k} l(t),
\end{equation*}
then it suffices to prove $M = 0.$
 (\ref{401}) implies that $M$ is non-negative and finite.
 Hence combining  (\ref{408}) and (\ref{409}), then using the Lebesgue dominated convergence theorem and  (\ref{410}), it finds
\begin{equation}\label{411}
M \leq C_{1} \epsilon_{0}(1+F(\delta)) M,
\end{equation}
where $F(\delta)$ is defined by
\begin{eqnarray*}
&&F(\delta)=
\epsilon_{0} \int_{0}^{\delta}(1-s)^{-\frac{N-\lambda}{q_{1}}} s^{-\alpha_{1}} \mathrm{d} s
+\epsilon_{0} \int_{0}^{\delta}(1-s)^{-\frac{N-\lambda}{q_{2}}} s^{-\alpha_{2}} \mathrm{d} s\nonumber\\
&&\qquad+\epsilon_{0} \int_{0}^{\delta}(1-s)^{-\frac{N-\lambda}{2 q_{1}}-\frac{N-\lambda}{2 q_{2}}} s^{-\frac{\alpha_{1}+\alpha_{2}}{2}} \mathrm{d} s\nonumber\\
&&\qquad+\epsilon_{0} \int_{0}^{\delta }(1-s)^{-\frac{1}{2}-\frac{N-\lambda}{2}\left(\frac{2}{q_{1}}-\frac{1}{q_{1}}\right)} s^{-\alpha_{1}}
+(1-s)^{-\frac{1}{2}-\frac{N-\lambda}{2}\left(\frac{2}{q_{2}}-\frac{1}{q_{1}}\right)} s^{-\alpha_{2}}\mbox{d}s\nonumber\\
&&\qquad+\epsilon_{0} \int_{0}^{\delta }(1-s)^{-\frac{1}{2}-\frac{N-\lambda}{2q_{1}}} s^{-\frac{\alpha_{1}+\alpha_{2}}{2}}
+(1-s)^{-\frac{1}{2}-\frac{N-\lambda}{2}\left(\frac{2}{q_{2}}-\frac{1}{q_{2}}\right)} s^{-\alpha_{2}} \mathrm{d} s\nonumber\\
&&\qquad+\epsilon_{0} \int_{0}^{\delta }(1-s)^{-\frac{1}{2}-\frac{N-\lambda}{2 q_{1}}} s^{-\frac{\alpha_{1}}{2}} d s
+\epsilon_{0} \int_{0}^{\delta }(1-s)^{-\frac{N-\lambda}{q_{2}}} s^{-\alpha_{2}} \mathrm{d} s\nonumber
\end{eqnarray*}
with

\begin{equation*}
\lim _{\delta \rightarrow 0} F(\delta)=0.
\end{equation*}
Hence, choosing $\epsilon_{0}$ and $\delta$
small enough and
using $(\ref{409}),$ we deduce $M=0$.
\end{proof}

\end{document}